\documentclass[a4paper,12pt]{article}
\usepackage[T1]{fontenc}
\usepackage[utf8]{inputenc}
\usepackage{lmodern,graphicx,bm,url,upgreek}
\usepackage{amsmath,amssymb,amsfonts,amsthm}
\usepackage{calc}
\usepackage{hyperref}
\usepackage{tikz}
\usetikzlibrary{decorations.markings}
\tikzstyle{vertex}=[circle, draw, inner sep=0pt, minimum size=6pt]
\newcommand{\vertex}{\node[vertex]}
\theoremstyle{plain}
     \newtheorem{thm}{Theorem}[section]
     
     \newtheorem{lem}[thm]{Lemma}

\theoremstyle{definition}
     \newtheorem{hyp}[thm]{Hypothesis}

     \newtheorem{rem}[thm]{Remark}
\usepackage[below]{placeins}

\newcommand{\PSp}{\ensuremath{\mathrm{PSp}}}
\newcommand{\PSL}{\ensuremath{\mathrm{PSL}}}
\newcommand{\PGL}{\ensuremath{\mathrm{PGL}}}

\newcommand{\Alt}{\ensuremath{\mathfrak{A}}}
\newcommand{\Sym}{\ensuremath{\mathfrak{S}}}

\newcommand{\Aut}{\ensuremath{\mathrm{Aut}}}

\newcommand{\Out}{\ensuremath{\mathrm{Out}}}

\newcommand{\DG}{\ensuremath{\Delta(G)}}

\newcommand{\Irr}{\ensuremath{\mathrm{Irr}}}
\newcommand{\cd}{\ensuremath{\mathrm{cd}}}

\newcommand{\PSU}{\ensuremath{\mathrm{PSU}}}

\title{\textbf{4-Regular prime graphs of nonsolvable groups}}
\author{Donnie Munyao Kasyoki$^\text{\tiny{1}}$, Paul Odhiambo Oleche$^\text{\tiny{1}}$\\
$^\text{\tiny{1}}$Department of pure and applied mathematics\\
Maseno University}
\date{}

\begin{document}

\maketitle
%\tableofcontents
\hrule\vspace*{0.5cm}

\noindent \textbf{Abstract}: Let $G$ be a finite group and $\cd(G)$ denote the character degree set for $G$. The prime graph $\DG$ is a simple graph whose vertex set consists of  prime divisors of elements in $\cd(G)$, denoted $\rho(G)$. Two primes $p,q\in \rho(G)$ are adjacent in $\DG$ if and only if $pq|a$ for some $a\in \cd(G)$. We determine which simple 4-regular graphs occur as prime graphs for some finite nonsolvable group.
\vspace*{1cm}


\vspace*{0.3cm}

\hrule

\section{Introduction}

The graphs arising from character degrees of finite groups have been studied extensively over the last few years. It was determined that the prime graphs of any finite group have diameter not exceeding 3 (see \cite{Wolf, LW, LW2}). In \cite{White}, D. White summarised the graph structure for the prime graphs of finite simple groups. These were classified on the basis of classification of finite simple groups. In that we are able to determine 	that most simple groups have a complete prime graph.

In \cite{HungReg}, H. P. Tong-Viet studied the 3-regular simple graphs that occur as prime graphs for some finite group. He proved that the complete cubic graph is the only 3-regular graph that occurs as a prime graph of some finite group $G$. C. P. M. Zuccari \cite{Zuccari2} obtained that the only noncomplete regular prime graphs for finite solvable groups with $n$ vertices can only be the $(n-2)$-regular, when $n$ is even. In particular, When $n$ is odd then no regular noncomplete graph occurs as a prime graph for some finite solvable group. However, the case for the regular graphs for nonsolvable groups has not been determined yet. In this paper, we seek to prove the following:

Let $\Gamma$ be a simple graph and let $x, y$ be two vertices of $\Gamma$. We write $x\sim y$ if $x$ is adjacent to $y$ in $\Gamma$. All simple groups considered are nonabelian.

\begin{thm}\label{thm:MAIN}
Let $G$ be a nonsolvable group and $\DG$ be a prime graph for $G$. If $\DG$ is 4-regular, then $\DG$ is complete with 5 vertices.
\end{thm}

\section{Simple groups}
$\Alt_n$ and $\Sym_n$ denoted the alternating and symmetric groups respectively.
\begin{lem}\textup{\cite{White}}\label{lem0}
Let $S$ be a simple group such that $\Delta(S)$ is connected. Then $\Delta(S)$ is complete except:
\begin{enumerate}
\item $S\cong J_1, M_{11}, M_{23}$
\item $S\cong \Alt_n, n\in \{5,6,8\}$
\item $S\cong\ ^2B_2(q^2), q^2=2^{2m+1}, m\geq 1$
\item $S\cong \PSL_3(q)$, $q$ a power of prime $p$ and $q-1\neq 2^i3^j,$ for some $ i\ge 1, j\ge 0$.
\item $S\cong \PSU_3(q^2), q$ is a power of  prime $p$ and $q+1\neq 2^i3^j$ for some $i,j\geq 0$
\end{enumerate}
\end{lem}

\begin{lem}Let $S$ be a finite simple group such that $\Delta(S)$ is $k$-regular. Then $\Delta(S)$ is complete or $k=0$.\end{lem}

\begin{proof}
If $\Delta(S)$ is disconnected, then by \cite[Lemma 2.6]{HungReg} $k=0$.

Thus we may assume that $\DG$ is connected.
By ATLAS \cite{ATLAS}, $S$ cannot be one of the groups in (1) or (2). 

If $S$ is one of the groups in (3), then by \cite[Theorem 3.3]{White3}, $\pi(S)=\{2\}\cup \pi(q^2-1)\cup\pi(q^4+1)$ and  the subgraph of $\Delta(S)$ induced by $\pi(S)\setminus\{2\}$ is complete and two is adjacent to precisely the primes in $\pi(q^2-1)$. There is no possibility of $\Delta(S)$ being regular since the primes in $\pi(q^2-1)$ are complete in $\Delta(S)$.

If $S$ is one of the groups in (4), It follows by \cite[Theorem 3.2]{White2} that if $q\neq 4$, then $\pi(S)=\{p\}\cup \pi((q-1)(q+1)(q^2+q+1))$ and the subgraph induced by $\pi((q-1)(q+1)(q^2+q+1))$ is complete and $p$ is adjacent to the primes dividing $q+1$ or $q^2+q+1$. So again $\Delta(S)$ has complete vertices. If $q=4$, by ATLAS \cite{ATLAS}, $\Delta(S)$ is not regular.

If $S$ is as in (5), then by \cite[Theorem 3.4]{White2}, $\pi(S)=\{p\}\cup \pi((q-1)(q+1)(q^2-q+1))$ and the subgraph induced by $\pi((q-1)(q+1)(q^2-q+1))$ is complete and $p$ is adjacent to the primes dividing $q-1$ or $q^2-q+1$. So again $\Delta(S)$ has complete vertices.
\end{proof}

\begin{lem}\label{lem:pent}
Let $S$ be a nonabelian simple group with $|\pi(S)|=5$ and  $\Delta(S)$ a subgraph of the house or the butterfly. Then $\Delta(S)$ is isomorphic to graphs in Figure~\ref{fig:p} and $S\cong \PSL_2(2^f)$, $f\geq 3$ or $\PSL_2(q), q$ a power of an odd prime $p$.

\begin{figure}[htb!]\centering
\begin{tikzpicture}
\vertex (f) at (4,0) {};
\vertex (g) at (5,0) {};
\vertex (h) at (5,1) {};
\vertex (i) at (4,1) {};
\vertex (j) at (4.5,2) {};

\vertex (a) at (7,0) {};
\vertex (b) at (8,0) {};
\vertex (c) at (7,2) {};
\vertex (d) at (8,2) {};
\vertex (e) at (7.5,1) {};
\path
	 (d) edge (c)
	 (a) edge (b)
	 (a) edge (e)
	 (b) edge (e)
	 (c) edge (e)
	 (d) edge (e)
	 (h) edge (i)
	 (g) edge (h)
	 (j) edge (i)
	 (j) edge (h)
	 (f) edge (i)
	 (f) edge (g);
	 \end{tikzpicture}
\caption{The house and the butterfly}
\label{fig:h}
\end{figure}
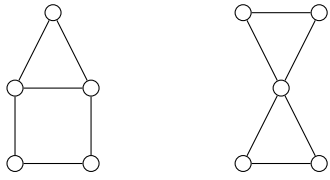\FloatBarrier

\begin{figure}[htb!]\centering
\begin{tikzpicture}
\vertex (a) at (1.5,2) {};
\vertex (b) at (1,0) {};
\vertex (c) at (2,0) {};
\vertex (d) at (1,1) {};
\vertex (e) at (2,1) {};
\path
		(e) edge (d)
		(c) edge (b);

\draw (1.5,-0.5) node{(a)} circle (0);
\draw (4.5,-0.5) node{(b)} circle (0);

\vertex (f) at (4,0) {};
\vertex (g) at (5,0) {};
\vertex (h) at (5,1) {};
\vertex (i) at (4,1) {};
\vertex (j) at (4.5,2) {};
\path
	 (h) edge (i)
	 (g) edge (h)
	 (j) edge (i)
	 (j) edge (h);
	 
\draw (7.5,-0.5) node{(c)} circle (0);

\vertex (k) at (7,0) {};
\vertex (l) at (8,0) {};
\vertex (m) at (8,1) {};
\vertex (n) at (7,1) {};
\vertex (o) at (7.5,2) {};
\path
	 (m) edge (n)
	 (o) edge (m)
	 (o) edge (n)
	 ;
\end{tikzpicture}
\caption{The prime graphs of $\PSL_2(2^6)$, $\PSL_2(5^3)$ and $\PSL_2(2^8)$}
\label{fig:p}
\end{figure}
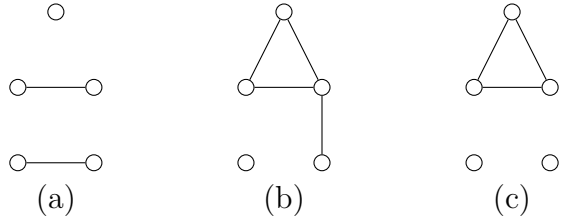\FloatBarrier
\end{lem}
\begin{proof}Clearly $\Delta(S)$ is either a house, a butterfly or is obtained by deleting at least one edge. Also, we note that every subgraph of the house or the butterfly is $K_4$-free.
Suppose that $\Delta(S)$ is connected. By \cite{ATLAS}, if $S$ is as in Lemma~\ref{lem0} (1) or (2), then $|\pi(S)|\neq 5$. 

If $S$ is as in (3), then by \cite[Theorem 3.4]{White3}, $S\cong\ ^2B_2(q^2), q^2=2^{2m+1}$. By \cite[Theorem 3.3]{White3} we have that $\Delta(S)$ has a subgraph isomorphic to $K_4$ and a vertex connected to some but not all primes in $V(K_4)$. 

Suppose that $S$ is as in (4) or (5). If $S\cong \PSL_{\ell+1}(q)$ or $S\cong \PSU_{\ell+1}(q^2), \ell\ge 2, q>2$ a power of a prime $p$, then $\Delta(S)$ is a $K_5$ or contains a $K_4$ by \cite[Theorem 3.2]{White2}. %If $S\cong \PSL_3(q), q\neq 4, q-1\neq 2^i3^i, i,j\geq 0$, $q$ a power of an odd prime, by \cite[Theorem 3.2]{White2}, $\Delta(S)$ is as in the previous paragraph. Thus $\Delta(S)$ is disconnected. 
 Thus $\Delta(S)$ is disconnected.

 Now, assume that $\Delta(S)$ has two connected components. Then by \cite[Theorem 3.1]{White2}, $S\cong \PSL_2(q), q>5$ a power of an odd prime $p$. Then the two connected components are $\{p\}$ and $\pi((q-1)(q+1))$. The component $\pi((q-1)(q+1))$ is complete  if and only if $q-1$ or $q+1$ is a power of 2, otherwise $\Delta(S)$ is as part  2(b) of \cite[Theorem 3.1]{White2}.  In this case, $\Delta(S)$ is disconnected with 2 connected components. We consider any subgraph of the two graphs with five vertices and two connected components (of course must have an isolated vertex). By \cite[Theorem 3.1]{White2}, $\Delta(S)$ will have an isolated vertex and a four-vertex component with a complete vertex since 2 is a complete vertex in the subgraph with vertices $\pi(q^2-1)$. This is as in Figure~\ref{fig:p}(b). For example, $$\cd(\PSL_2(5^3))=\{1,5^3-1,5^3+1,5^3,(5^3+1)/2\}.$$ Observe that $\Delta(\PSL_2(5^3))$ is isomorphic to the graph in Figure~\ref{fig:p}(b). 

Now we may suppose that $\Delta(S)$ has three connected components. Then $S\cong \PSL_2(2^f)$, $f\ge 3$ so that $\Delta(S)$ has three connected components with at least one isolated vertex. The remaining component(s) is (are) complete. The only possibilities are as in Figure~\ref{fig:p}(a) and (c).
An example for graph (a) is when $q=2^6$. By \cite{ATLAS} we have  that $$\cd(\PSL_2(2^6))=\{1,2^6,2^6-1,2^6+1\}=\{1,2^6,5\cdot 13, 3^2\cdot 7\}.$$  In this case we obtain the graph in Figure~\ref{fig:p}(a). 
%
%Lastly, observe that if $\Delta(S)$ is connected then $\Delta(S)$ is not $K_4$-free.
\end{proof}

The following results will be used throughout this paper.

\begin{lem}\textup{\cite[Lemma 4.2]{TVtri}}\label{TVtri}
Let $N$ be a normal subgroup of a group $G$ such that $G/N\cong S$, where $S$ is a nonabelian simple group. Let $\theta\in \Irr(N)$. Then either $\psi(1)/\theta(1)$ is divisible by two distinct in $\pi(S)$ for some $\psi\in \Irr(G|\theta)$ or $\theta$ extends to a $\theta_0\in \Irr(G)$ and $S\cong \Alt_5$ or $\PSL_2(8)$.
\end{lem}

The following result will be referred to as the Gallagher's Theorem.

\begin{thm}\textup{\cite[Corollary 6.17]{Isaacs}}\textup{[Gallagher's Theorem]}
Let $N$ be a normal subgroup of a group $G$ and let $\theta$ be an irreducible character of $N$. If $\theta$ is extendible to $G$, then the character $\theta\psi$ for $\psi\in \Irr(G/N)$ are irreducible, distinct for distinct $\psi$ and are all of the irreducible constituents of $\theta^G$.
\end{thm}

\begin{thm}\textup{\cite[Theorem 11.7]{Isaacs}}\label{TSchur}
Let $G$ be a group and $N\unlhd G$. Let $\theta\in \Irr(N)$ be invariant in $G$. If the Schur multiplier of $G/N$ is trivial, then $\theta$ extends to $G$.
\end{thm}

In \cite{Hupp}, Huppert determined all the simple groups whose orders are divisible by four primes. The results were summarized in Table 2 and Table 3 of \cite{Hupp}. The groups that were not featured in the tables then will be one of the groups listed in the Lemma below:

\begin{lem}\label{four}
Let $S$ be a finite nonabelian simple group with $|\pi(S)|=4$ and $\pi(S)\neq \pi(G)$, where $G$ is a group in \cite[Table 2,3]{Hupp}, then exactly one of the following occurs:
\begin{enumerate}
\item[(i)] $S\cong \PSL_2(r)$ where $r=\max(\pi(S))$
\item[(ii)] $S\cong \PSL_2(2^s)$ for some prime $s$ such that $2^s-1=\max(\pi(S))$ is a Mersenne prime;
\item[(iii)] $S\cong \PSL_2(3^t)$ for some prime $t\ge 5$ such that $\pi(S)=\{2,3,p,q\}$
\end{enumerate}
\end{lem}

Let $S$ be one of the groups in Lemma~\ref{four} above. Since the outer automorphism of $S\cong \PSL_2(p^f),$ ($p$ a prime and $f\ge 1$ an integer) is of order $(2,p-1)\cdot f$. With this in mind, we deduce that $|\pi(\Out(S))\setminus \pi(S)|\le 1$. 

\section{Nonsolvable groups}

\begin{lem}\label{lem:join}
Let $G$ and $H$ be groups and $\mathfrak{D}:=\rho(G)\cap\rho(H)$. Suppose that $\rho(H)\setminus \mathfrak{D}$ spans a complete subgraph of $\Delta(H)$. Then  $\Delta(H)$ is a complete subgraph of $\Delta(G\times H)$. Moreover, $\Delta(G\times H)$ contains at 	least $|\rho(H)|$ complete vertices.
\end{lem}

\begin{proof}
Clearly $\mathfrak{D}$ spans a complete subgraph of $\Delta:=\Delta(G\times H)$. It suffices to show that $\rho(H)-\mathfrak{D}$ is adjacent to every prime in $\mathfrak{D}$. Since $\mathfrak{D}\subseteq \rho(G)$, we must have $\lambda_i(1)\in \cd(G)$ for $i=1,2,\ldots, |\mathfrak{D}|$, not necessarily distinct such that $p_i|\lambda_i(1)$ for each $i$, all $p_i$'s in $\mathfrak{D}$. Also, for each $q_i\in \rho(H)\setminus \mathfrak{D}$ there is a $\theta_j\in \Irr(H)$ such that $q_j|\theta_j(1)$, $j=1,2,\ldots,|\rho(H)\setminus \mathfrak{D}|$. Since $\cd(G\times H)=\{\eta(1)\vartheta(1) |\eta(1)\in \cd(G), \vartheta(1)\in \cd(H)\}$, we must have that $\theta_j(1)\lambda_i(1)\in \cd(G\times H)$ for each $i$ and each $j$. In fact we can observe that the primes in $\mathfrak D$ are complete vertices of $\Delta(G\times H)$. 

To see the second part, Let $p\in \rho(H)\setminus \mathfrak D$. We show that it is adjacent to every other prime in $\rho(G)$. Let $p|\lambda(1)$ for some $\lambda\in \Irr(H)$, then $\theta (1)\lambda(1)\in \cd(G\times H)$ for every $\theta\in \Irr(G)$.

\end{proof}

\begin{lem}\label{lem:1}
Let $G$ be a nonsolvable group and $\DG$ be 4-regular with %$K_5$-free, of maximal degree $d\leq 4$ with 
more than 5 vertices.  Then every nonsolvable chief factor of $G$ is  simple. % and its prime graph is noncomplete.% or $\Delta(S)$ has three vertices and $\DG$ has 6.
\end{lem}

\begin{proof} Suppose that $\DG$ contains a subgraph isomorphic to $K_5$. Then since $\DG$ is  4-regular, we must have that $\DG$ is isomorphic to $K_5$ or is a disjoint union of $K_5$'s. 
This contradicts \cite[Lemma 2.6]{HungReg}. We may assume that $\DG$ is $K_5$-free. Let $M/N=S_1\times \cdots \times S_k$, where $S_i\cong S$ a nonabelian simple group, be a chief factor of $G$ and $C/N=C_{G/N}(M/N)$. Then $G/C$ contains a minimal subgroup isomorphic to $S^k$. We show that $k=1$. By a way of contradiction, assume that $k\ge 2$. Let $L\leq MC$ be such that $L/C\cong S$. Since $G/C$ has no nontrivial abelian normal subgroups we have that $\pi(G/C)=\rho(G/C)$. By \cite[Main Theorem]{LewisCh} we have that $\Delta(G/C)$ is complete if $k\geq 2$ hence $|\pi(G/C)|\leq 4$. We must have that $\rho(G/C)=\{p_i\}_{i=1}^n$ where $n=3,4$.
\vspace{0.3cm}

\noindent{\bf Case 1: $|\rho(G/C)|=4$} %\acute{e}\'e

%\begin{proof}
Suppose that $x\in \rho(G)\setminus \rho(G/C)=:\rho$. It follows that $x\in \rho(C)$. Then there exists $\theta\in \Irr(C)$ such that $x\in \pi(\theta(1))$. By Lemma~\ref{TVtri}, either $\theta$ extends to $\theta_0\in \Irr(L)$ or $\psi(1)/\theta(1)$ is divisible by two distinct primes in $\pi(L/C)$ for some $\psi\in \Irr(L|\theta)$. We consider the two cases separately:\vspace{0.23cm}

\noindent{\bf Subcase 1: $\theta$  extends to $L$}\vspace*{0.23cm}

If this holds then all the primes in $\rho(S)$ are adjacent to $x$. Thus if $|\rho(S)|=4$, then $\DG$ will contain a $K_5$, thus we must have that $|\rho(S)|=3$. Since $|\rho(G)|\geq 6$, there is a prime $x\neq y\in \rho$. So we have that $y$ is adjacent to all primes in $\pi(S)$ or $y$ is adjacent to two primes in $\pi(S)$. If the former occurs then the primes in $\pi(S)$ will have degree $\geq 5$. Assume that the latter occurs. Then $y$ is adjacent to two primes in $\pi(S)$. This two primes will have degree $\geq 5$.\vspace*{0.3cm}

\noindent{\bf Subcase 2: $x$ adjacent to two primes in $\rho(S)$}\vspace*{0,23cm}

 If this occurs then $x$ is adjacent to two primes in $\pi(L/C)$, say $p_1, p_2$ which implies that $\{x,p_1,p_2\}$ forms a triangle. Since $|\rho(G)|\geq 6$, there is another prime $x\neq y\in \rho$ such that $y$ is adjacent to all primes in $\pi(S)$ or $y$ is adjacent to two primes in $\pi(S)$. If the former occurs, then we have two primes with degree greater than or equal to 5. Thus we may assume that the latter occurs. In this case we must have that $|\rho(G)|\leq 6$ or else we can find a different prime in $\rho$ and obtain a contradiction. If $|\rho(G)|=6$, then we require that the neighbours of $y$ be different from those of $x$.  In this case we have a graph with 6 vertices and contains a $K_4$. This graph cannot be 4-regular. \vspace*{0.23cm}
 
%\end{proof}

\noindent\textbf{Case 2: $|\rho(G/C)|=3$}\vspace*{0.23cm}

Define $\delta:=\rho(G)\setminus \pi(S)$ and observe that if $|\rho(G)|\ge 7$, then $|\delta|\geq 4$. Using the arguments used above. Each prime in $\delta$ is either adjacent to all primes in $\pi(S)$ or is adjacent to 2 primes in $\pi(S)$. Either way, we observe that there is at least one prime in $\pi(S)$ with degree $\geq 5$. The proof is now complete. 

Suppose that $|\rho(G)|=6$ such that $\delta=\{x,y,z\}$ and $\pi(S)=\{p_i\}_{i=1}^3$. Notice that $\delta\in \rho(C)$. If  $C$ is solvable then $\delta$ must span at least an edge, say $x\sim y$. Let $\vartheta\in \Irr(C)$ be such that $xy|\vartheta(1)$. By Lemma~\ref{TVtri} we have that $\vartheta$ extends to $\vartheta_0\in \Irr(L)$ or $\psi(1)/\vartheta(1)$ is divisible by two distinct primes in $\pi(S)$. In either way we obtain a contradiction. So, we may assume that $C$ is nonsolvable. If $\delta$ spans at least an edge, then by previous argument we obtain a contradiction. Suppose that it doesn't span an edge, it follows by \cite[Theorem 4.1]{LewisCo} that $C\cong \PSL_2(2^f)\times A$ for some $f\geq 2$ and $A$ abelian. This implies that $|\rho(C)|=|\pi(\PSL_2(2^f))|>3$ since $2\in \pi(S)\cap \pi(\PSL_2(2^f))$. Since 2 is an isolated vertex for $\Delta(\PSL_2(2^f))$, it implies either $\Delta(C)$ has 4 disconnected components, or $\delta$ spans at least an edge, contradicting our assumption. This final contradiction implies that $k=1$.

%
%
%Suppose that $\DG$ is the 4-regular graph of 6 vertices, then we see that $\delta$ spans a triangle and thus by  Lemma~\ref{TVtri} we have a complete cubic subgraph.  This contradiction completes the proof.
\end{proof}

\begin{lem}\label{lem:2}Let $G$ satisfy the hypothesis of Lemma~\ref{lem:1}, with $6\le |\rho(G)|\le 9$. Let $N\unlhd G$ be the solvable radical for $G$. If $M/N$ is a chief factor of $G$, then $G/N$ is almost simple with socle $M/N$. % or $|rho(G)|=6$ and $|\pi(M/N)|=3$. 
\end{lem}

\begin{proof}
Let $C/N=C_{G/N}(M/N)$. Then $G/C$ is almost simple with socle $MC/C\cong M/N$. By Lemma~\ref{lem:1}, $M/N\cong S$ is simple nonabelian and thus $C\cap M=N$. %We can find a subgroup $L$ such that $N<L\leq C$ such that $L/N$ is a chief factor for $G$. Again, $L/N\cong T$, a nonabelian simple group. 
It suffices to show that $C=N$. By a way of contradiction suppose that $C\neq N$. Let $L\le C$ be such that $L/N$ is a minimal normal subgroup of $G/N$. Then we have that $L/N$ is a nonabelian simple group. Observe that $\pi(L/N)\subseteq \pi(C/N)$ and $CM/N\cong C/N\times M/N$, Which implies that the primes in $\rho(C/N)\cap \rho(M/N)$ are complete vertices in the subgraph induced by $\rho(C/N)\cup \rho(M/N)$. Observe that since $2\in \rho(C/N)\cap \rho(M/N)$, we must have that $|\rho(C/N)\cup \rho(M/N)|\leq 5$.  Otherwise $\deg(2)\geq 5$. We deduce that $3\leq|\rho(C/N)\cup \rho(M/N)|\leq 5$. Write $\mathfrak{F}:=\rho(C/N)\cup \rho(M/N)$.\vspace*{0.23cm}

\noindent\textbf{\bf Case 1: $|\mathfrak{F}|=5$}\vspace*{0.23cm}

Let $\mathfrak E=\rho(C/N)\cap \rho(M/N)$. By the structure of the graphs in Figures~\ref{fig:1}-\ref{fig:4}, any five vertices chosen in any of the graphs contains at most 2 complete vertices. Then this forces $|\mathfrak E|\le 2$. We observe that if $|\rho(C/N)|>|\rho(L/N)|$, then $|\rho(C/N)|\ge 4$. In this case we must have that $|\pi(M/N)|=3$, since otherwise we would have that $|\mathfrak E|\ge 3$ contradicting the structure of the graphs in consideration. %\vspace*{0.23cm}

If we suppose that $|\rho(C/N)|\ge 4$ and $|\pi(M/N)|=3$, then we must have that $|\mathfrak E|=2$. By invoking Lemma~\ref{lem:join}, we observe that $\mathfrak F$ spans a subgraph with 3 complete vertices. This again contradicts the stucture of the graphs in consideration. Thus the only cases we need to consider are when both $\rho(C/N)$ and $\pi(M/N)$ have 3 vertices and so it suffices to assume that $|\rho(C/N)|=|\pi(L/N)|$.

Assume first that $\DG$ is $K_4$-free. By \cite[Theorem B]{HungK4}, $|\rho(G)|\leq 7$. There are two 4-regular graphs of order 7 shown in Figure~\ref{fig:1} and \ref{fig:2} and one with 6 vertices.  Observe that for any 5 vertices chosen in Figure~\ref{fig:1}, \ref{fig:2} and \ref{fig:6vert}, we have that $|\rho(C/N)\cap \rho(M/N)|\leq 1$, $|\rho(C/N)\cap \rho(M/N)|\leq 2$ and  $|\rho(C/N)\cap \rho(M/N)|\leq 1$ respectively. %Suppose we find simple groups $T$ and $S$ with $|\pi(S)|=3$, $|\pi(T)|=4$, $|\pi(S)\cap\pi(T)|=2$ and $|\pi(S)\cup \pi(T)|=5$. By Lemma~\ref{lem:join}, $\Delta(S\times T)$ must contain 3 complete vertices. However, this is not the case in Figure~\ref{fig:2}. So it suffices to look at the case when $|\rho(C/N)\cap \rho(M/N)|\leq 1$.

We need to find simple groups $T$ and $S$ with $|\pi(S)|=|\pi(T)|=3$, $|\pi(S)\cap\pi(T)|=1$ and $|\pi(S)\cup \pi(T)|=5$. By \cite{Hupp}, the only simple groups whose orders are divisible by 3 primes only  are $$\PSU_3(3)\cong \PSp_4(2), \PSL_2(7), \PSL_2(8), \PSL_2(17), \PSL_3(3), \Alt_5, \Alt_6.$$
Each of these groups have order divisible by both 2 and 3. This contradicts the fact that $|\rho(S)\cap\rho(T)|=1$. If $|\mathfrak E|=2$ the we obtain that $|\mathfrak F|=4$, a contradiction.

%Now we consider the case when $\DG$ is isomorphic to \ref{fig:2}. Observe that $|G/N|=5$, this is because $|\pi(G/C)|=3$ while that of $|\pi(C/N)|=3$

%
%\noindent{\bf Case 1: $S$ or $T$ is $\PSU_3(3)$}
%
%We choose 5 vertices such that $\mathfrak{J}:=\rho(S)\cup \rho(T)$ has 5 elements. We have that $\Delta(PSU_3(3))$ is a triangle with vertices $\{2,3,7\}$.  Then $\{2,5,k\}=\rho(T)$ with $k\neq 3,7$. Observe that $\DG$ must contain a $K_4$ with vertices $\{2,3,7,5\}$ since $L/N\times M/N\unlhd G/L$. Clearly the graph in Figure~\ref{fig:1} is $K_4$-free.
%
%
%\noindent{\bf Case 2: $S\cong \PSL_2(7)$ or $\PSL_2(17)$}
%
%In this case $\Delta(S)$ is disconnected with an isolated vertex  and $2\sim 3$. Observe that $\rho(T)$ must disconnected with 3 components, or have a single edge with a prime adjacent to 2. Otherwise $\DG$ will contain a $K_4$. In the latter case the subgraph induced by $\mathfrak{F}$ is not isomorphic to the subgraph chosen.
%
%\noindent{\bf Case 3: $S\cong \PSL_2(8)$}
%
%In this case $\Delta(S)$ is disconnected with 3 components $\{2\}, \{3\}, \{7\}$. Observe that no simple group $T$ will result in $\Delta(S\times T)$ isomorphic to the subgraph chosen.
%
%\noindent{\bf Case 4: $S\cong \PSL_3(3)$}   
%
%In this case we have $\Delta(S)$ is a path. Also we obtain that no simple group $T$ with $\Delta(S\times T)$ isomorphic to the subgraph chosen.
%
%

\begin{figure}[h!]\centering
\begin{tikzpicture}
	\vertex (w) at (90-360/7:1.4) [label=right:$p_3$] {};
	\vertex (u) at (90:1.4) [label=above:$p_2$] {};
	\vertex (v) at (90+360/7:1.4) [label=left:$p_1$] {};
	\vertex (x) at (90+720/7:1.4) [label=left:$q_1$] {};
	\vertex (y) at (90+1080/7:1.4) [label=left:$q_2$] {};
	\vertex (a) at (90+1440/7:1.4) [label=right:$q_3$] {};
	\vertex (b) at (90+1800/7:1.4) [label=right:$q_4$] {};
	\path 
 		(w) edge (u)
 		(u) edge (v)
 		(v) edge (x)
 		(y) edge (a)
 		(b) edge (a)
 		(x) edge (y)
 		(b) edge (w)
 		(v) edge (w)
 		(u) edge (x)
 		(v) edge (y)
 		(x) edge (a)
 		(y) edge (b)
 		(a) edge (w)
 		(b) edge (u)
	;
\end{tikzpicture} 
\caption{4-regular graph of order 7 with 7 triangles}
\label{fig:1}
\end{figure}
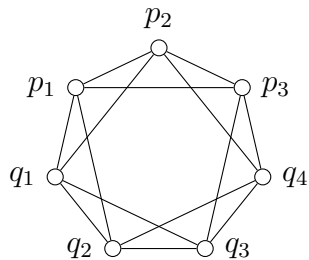\FloatBarrier

\begin{figure}[htb!]\centering
\begin{tikzpicture}
	\vertex (w) at (90-360/7:1.4) [label=right:$p_3$] {};
	\vertex (u) at (90:1.4) [label=above:$p_2$] {};
	\vertex (v) at (90+360/7:1.4) [label=left:$p_1$] {};
	\vertex (x) at (90+720/7:1.4) [label=left:$q_1$] {};
	\vertex (y) at (90+1080/7:1.4) [label=left:$q_2$] {};
	\vertex (a) at (90+1440/7:1.4) [label=right:$q_3$] {};
	\vertex (b) at (90+1800/7:1.4) [label=right:$q_4$] {};
	\path 
 		(w) edge (u)
 		(u) edge (v)
 		(v) edge (x)
 		(y) edge (a)
 		(b) edge (a)
 		(x) edge (y)
 		(b) edge (w)
 		(v) edge (w)
 		(v) edge (b)
 		(u) edge (y)
 		(u) edge (a)
 		(x) edge (w)
 		(x) edge (a)
 		(b) edge (y)
	;
\end{tikzpicture}
\caption{4-regular graph of order 7 with 6 triangles}
\label{fig:2}
\end{figure}
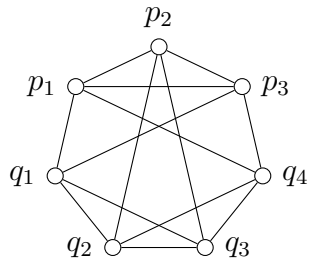\FloatBarrier

%\begin{figure}[htb!]\centering
%\begin{tikzpicture}
%\draw[fill] (6,0) circle (0.1);
%\end{tikzpicture}
%\end{figure}  

We may now assume that $\DG$ contains a $K_4$. So $|\rho(G)|\geq 8$. First suppose $|\rho(G)|=8$. We have only one 4-regular graph of order 8 with a $K_4$ as a subgraph. Consider Figure~\ref{fig:3} below, any choice of 5 vertices will induce a subgraph with only one complete vertex. This implies that $\pi(S)\cap\pi(T)=\{1\}$. Again if $|\pi(S)|=|\pi(T)|=3$, we obtain a contradiction as above. %Thus we may assume that one of the groups will be divisible by 4 primes. Since $|\pi(S)\cup\pi(T)|=5$ then we must have that $|\pi(S)\cap\pi(T)|=2$. By Lemma~\ref{lem:join}, $\Delta(S\times T)$ must contain three complete vertices. However, any choice of five vertices in the graph in Figure~\ref{fig:3} contains only one complete vertex. 

\begin{figure}[h!]\centering
\resizebox{5cm}{!}{\begin{tikzpicture}
\vertex (a) at (0,0) [label=below:$p_4$] {};
\vertex (e) at (4,0) [label=below:$q_4$] {};
\vertex (c) at (2,1) [label=above:$p_2$] {};
\vertex (b) at (2,-1) [label=below:$p_3$] {};
\vertex (d) at (0,2) [label=above:$p_1$] {};
\vertex (g) at (6,1) [label=above:$q_2$] {};
\vertex (f) at (6,-1) [label=below:$q_3$] {};
\vertex (h) at (4,2) [label=above:$q_1$] {};
\path
		(a) edge (e)
		(b) edge (f)
		(d) edge (h)
		(c) edge (g)
		(a) edge (b)
		(a) edge (c)
		(a) edge (d)
		(b) edge (c)
		(b) edge (d)
		(c) edge (d)
		(e) edge (f)
		(e) edge (g)
		(e) edge (h)
		(f) edge (g)
		(f) edge (h)
		(g) edge (h);
\end{tikzpicture} }
\caption{4-regular graph with a $K_4$}\label{fig:3}
\end{figure}
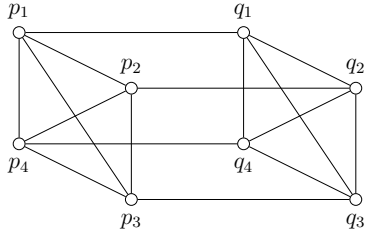\FloatBarrier

Lastly we consider the case when $\DG$ has 9 vertices.  We have only two 4-regular graphs which contains at least one $K_4$. The graphs are shown in Figure~\ref{fig:4}.

However, any subgraph of order 5 chosen from the graph in Figure~\ref{fig:4}(b) has only one complete vertex. This obtains a contradiction by previous arguments.
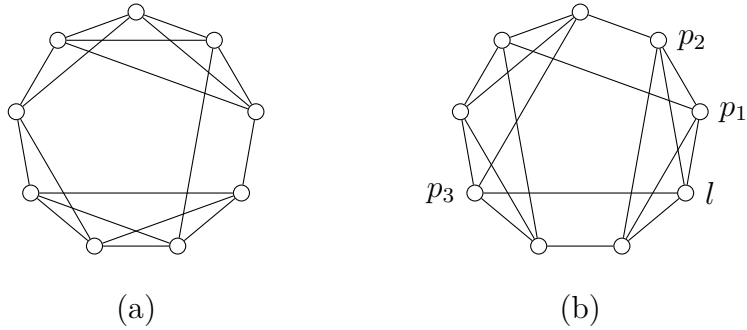
\begin{figure}[htb!]\centering
\begin{tikzpicture}
	\vertex (w) at (10:1.6) {};
	\vertex (u) at (50:1.6) {};
	\vertex (v) at (90:1.6) {};
	\vertex (x) at (130:1.6) {};
	\vertex (y) at (170:1.6) {};
	\vertex (a) at (210:1.6) {};
	\vertex (b) at (250:1.6) {};
	\vertex (c) at (290:1.6) {};
	\vertex (d) at (330:1.6) {};
	\path 
 		(d) edge (a)
 		(d) edge (c)
 		(d) edge (w)
 		(d) edge (b)
 		(w) edge (u)
 		(w) edge (v)
 		(w) edge (x)
 		(u) edge (x)
 		(u) edge (c)
 		(v) edge (y)
 		(b) edge (y)
 		(v) edge (x)
 		(x) edge (y)
 		(a) edge (y)
 		(b) edge (a)
 		(u) edge (v)
 		(b) edge (c)
 		(a) edge (c)
	;
\draw (0,-2) node [below] {(a)} circle (0);
\end{tikzpicture}\hspace*{2cm}\begin{tikzpicture}
	\vertex (w) at (10:1.6) [label=right:$p_1$] {};
	\vertex (u) at (50:1.6) [label=right:$p_2$] {};
	\vertex (v) at (90:1.6) {};
	\vertex (x) at (130:1.6) {};
	\vertex (y) at (170:1.6) {};
	\vertex (a) at (210:1.6) [label=left:$p_3$] {};
	\vertex (b) at (250:1.6) {};
	\vertex (c) at (290:1.6) {};
	\vertex (d) at (330:1.6) [label=right:$l$] {};
	\path 
 		(d) edge (u)
 		(d) edge (c)
 		(d) edge (w)
 		(d) edge (a)
 		(w) edge (u)
 		(a) edge (v)
 		(b) edge (x)
 		(w) edge (x)
 		(u) edge (c)
 		(v) edge (y)
 		(b) edge (y)
 		(v) edge (x)
 		(x) edge (y)
 		(a) edge (y)
 		(b) edge (a)
 		(u) edge (v)
 		(b) edge (c)
 		(w) edge (c)
	;
\draw (0,-2) node [below] {(b)} circle (0);
\end{tikzpicture}
\caption{4-regular graphs of order 9 with $K_4$}
\label{fig:4}
\end{figure}\FloatBarrier

Now,  consider the graph in Figure~\ref{fig:4}(a). We can choose a subgraph with 5 vertices and 2 complete vertices. We obtain a subgraph isomorphic to the graph in Figure~\ref{fig:5}.

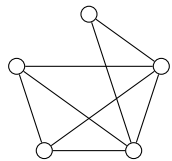
\begin{figure}[htb!]\centering
\begin{tikzpicture}
	\vertex (w) at (18:1) {};
	\vertex (u) at (90:1) {};
	\vertex (v) at (162:1) {};
	\vertex (x) at (234:1) {};
	\vertex (y) at (306:1) {};
	\path 
 		(x) edge (y)
 		(x) edge (v)
 		(v) edge (w)
 		(y) edge (w)
 		(y) edge (v)
 		(y) edge (u)
 		(u) edge (w)
 		(x) edge (w)
 		%(x) edge (u)
	;
\end{tikzpicture}
\caption{Subgraph with 2 complete vertices}
\label{fig:5}
\end{figure}\FloatBarrier

We seek to obtain simple groups $S$ and $T$ such that $\Delta(S\times T)$ is isomorphic to graph in Figure~\ref{fig:5}. Since $|\pi(S)\cap\pi(T)|=2$, one of the simple groups has order divisible by 4 primes. We have shown that this cannot occur.

%Since $|\pi(S)\cap\pi(T)|=2$ and one of the sets has 3 vertices, by Lemma~\ref{lem:join}, then subgraph induced by $\pi(S)\cup\pi(T)$ must have at least 3 complete vertices. Clearly the graph in Figure~\ref{fig:5} has only 2 complete vertices, a contradiction. \vspace*{0.23cm}

\noindent\textbf{\bf Case 2: $|\mathfrak{F}|=4$}\vspace*{0.23cm}
%By Lemma~\ref{lem:join}, if $|\mathfrak F|=4$, then we will have a $K_4$. Also, w

 It follows that $2\le |\pi(M/N)\cap\pi(C/N)|\le 4$. Since $3\le |\pi(M/N)|, \pi(C/N)|\le 4$, it will suffice to assume that $|\rho(C/N)|=4$ and $|\pi(M/N)|=4$ and $ |\pi(M/N)\cap\pi(C/N)|= 4$. We may also assume that $|\pi(G/C)|\ge|\pi(M/N)|$ and $|\pi(C/N)|\ge |\pi(L/N)|$ so that we have that $|\rho(G/N)|\le 5$  by \cite{Hupp,GAP}.  By Lemma~\ref{lem:join}, the subgraph induced by $\mathfrak{F}$ is a complete cubic graph no matter the choice of the number of primes in each of the two cases. Now, the 4-regular graphs with  six and seven vertices are $K_4$-free and thus there is nothing to prove. We may assume that $|\rho(G)|\geq 8$.

If $\DG$ is the 4-regular graph in Figure~\ref{fig:3}, without loss of generality, let $\{p_i\}_{i=1}^4\cup \{q_k\}=\rho(G/N)$ for some fixed $k\in \{1,2,3,4\}$, say $k=1$. Let $q_j\notin \rho(G/N), j\in \{2,3,4\}$, then we must have  $q_j|\theta(1)$ for some $\theta\in \Irr(N)$. By Lemma~\ref{TVtri}, $\theta$ extends to $L$ (or $M$) or $\psi(1)/\theta(1)$ is divisible by two primes in $\pi(M/N)$  for some $\psi\in \Irr(M|\theta)$. This implies that each $q_j$ is adjacent to at least 2 distinct primes in $\mathfrak{F}$. This is not the case for the graph in Figure~\ref{fig:3}. A similar argument obtains a contradiction for the graphs in Figure~\ref{fig:4}(a) and (b).\vspace*{0.23cm}

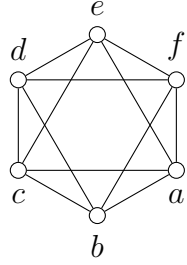
\begin{figure}[htb!]\centering
\begin{tikzpicture}
	\vertex (f) at (30:1.2)[label=above:$f$]{};
	\vertex (e) at (90:1.2)[label=above:$e$]{};
	\vertex (d) at (150:1.2)[label=above:$d$]{};
	\vertex (c) at (210:1.2)[label=below:$c$]{};
	\vertex (b) at (270:1.2)[label=below:$b$]{};
	\vertex (a) at (330:1.2)[label=below:$a$]{};
	\path 
		(b) edge (c)
		(b) edge (a)
		(b) edge (d)
		(b) edge (f)
		(c) edge (d)
		(c) edge (e)
		(c) edge (a)
		(d) edge (e)
		(d) edge (f)
		(e) edge (f)
		(e) edge (a)
		(f) edge (a)
	;
\end{tikzpicture}
%\begin{tikzpicture}
%\draw[fill] (0,0) node[below]{$a$} circle (0.1);
%\draw[-] (-1,1.73205080756888)--(0,3.46410161513775);
%\draw[-] (2,0)--(2,3.46410161513775);
%\draw[-] (2,0)--(0,0);
%\draw[-] (2,3.46410161513775)--(0,3.46410161513775);
%\draw[-] (0,0)--(-1,1.73205080756888);
%\draw[-] (2,3.46410161513775)--(-1,1.73205080756888);
%\draw[-] (2,3.46410161513775)--(3,1.73205080756888);
%\draw[-] (2,0)--(3,1.73205080756888);
%\draw[-] (0,0)--(3,1.73205080756888);
%\draw[-] (0,3.46410161513775)--(3,1.73205080756888);
%\draw[-] (0,3.46410161513775)--(0,0);
%\draw[-] (2,0)--(-1,1.73205080756888);
%\draw[fill] (-1,1.73205080756888) node[left]{$b$} circle (0.1);
%\draw[fill] (2,0) node[below]{$f$} circle (0.1);
%\draw[fill] (3,1.73205080756888) node[right]{$e$} circle (0.1);
%\draw[fill] (0,3.46410161513775) node[above]{$c$} circle (0.1);
%\draw[fill] (2,3.46410161513775) node[above]{$d$} circle (0.1);
%\end{tikzpicture}
\caption{The 4-regular graph with 6 vertices}
\label{fig:6vert}
\end{figure}\FloatBarrier

\noindent\textbf{\bf Case 3: $|\mathfrak{F}|=3$}\vspace*{0.23cm}

By \cite{Hupp}, $M/N$ is a simple group with $\pi(|\mathrm{Out}(M/N)|)\subset \pi(M/N)$ and it follows by \cite{ATLAS} that $\pi(G/C)=\pi(MC/C)=\pi(M/N)$. Also $\pi(C/N)=\pi(L/N)$ and thus $\rho(G/N)=\mathfrak{F}$. The subgraph induced by $\mathfrak{F}$ is a triangle. Consider the graphs in Figures~\ref{fig:1} - \ref{fig:4}. Choose any two prime $q\in \rho(G)\setminus \mathfrak{F}$. We obtain  by previous arguments that $q$ is adjacent to at least two primes in $\mathfrak{F}$. This condition is not satisfied by all vertices $q\in \rho(G)\setminus \mathfrak{F}$ in any of the graphs chosen apart from one with 6 vertices. Now suppose that $\DG$ is isomorphic to Figure~\ref{fig:6vert}. If we choose any triangle in the graph, we obtain that the remaining primes are all connected. Let $r, l\in \rho(G)-\mathfrak F$. Then $r,l\in \rho(N)$. Let $\gamma\in \Irr(N)$ be such that $rl|\gamma(1)$. By Lemma~\ref{TVtri}, we have that $\varphi(1)/\gamma(1)$ is divisible by two distinct primes in $\pi(M/N)$ or $\gamma$ extends to $M$. Suppose that the former occurs, then we obtain a subgraph isomorphic to $K_4$, a contradiction so we may assume that the latter occurs. In this case we obtain that $\DG$ contains a subgraph with 5 vertices and two complete vertices. This does not occur as a subgraph of Figure~\ref{fig:6vert}. We obtain the final contradiction and thus $C=N$. 
\end{proof}

\begin{hyp}\label{hyp}
Let $G$ be nonsolvable group whose prime graph $\DG$ is 4-regular with $7\le |\rho(G)|\le 9$. Let $N\unlhd G$ be the solvable radical of $G$ such that $G/N$ is almost simple with socle $M/N\cong S$ a nonabelian simple group.
\end{hyp}

\begin{lem}\label{lem:219}
Let $N\unlhd G$ be groups so that $G/N\cong J_1$. Let $p\in\rho(G)$ be such that $p\notin \pi(G/N)$. Then $p$ is connected to all primes in $\pi(J_1)$ or is connected to either 2 or 19.
\end{lem}

\begin{proof} Let $\chi\in \Irr(G)$ be such that $p|\chi(1)$.  Let $\theta$ be an irreducible constituent for $\chi_N$. Let $T=I_G(\theta)$, and let $\psi \in \Irr(T)$ be the Clifford correspondent lying between $\theta$ and $\chi$. If $T = G$, then $\theta$ extends to  $G$ since the Schur multiplier of $J_1$ is trivial by \cite[Theorem 11.7]{Isaacs}. By  Gallagher's Theorem we have that $\theta(1)\vartheta(1)\in \cd(G)$  for every $\vartheta\in \Irr(G/N)$. It follows that $p$ is connected to all the primes in $\pi(J_1)$. We may assume that $T < G$. Then it follows that $T\leq H$ for some maximal subgroup $H$ of $G$. It is easy to see that $H/N$ is a maximal subgroup of $G/N\cong J_1$. According to the Atlas \cite{ATLAS},
the possibilities for $|G : H |$ are 266, 1045, 1463, 1540, 1596, 2926, and 4180. We have that  $|G:T|\psi(1)\in \cd(G)$.  Observer that $|G : H |$ divides $|G : T |$. Each of the possibility of $|G:H|$ contains a 2 or a 19 as a prime divisor.
\end{proof}

\begin{lem}\label{simple3primes}
Assume Hypothesis~\ref{hyp}. If $|\pi(S)|=3$, then $G$ does not exist. \end{lem}

\begin{proof}
It suffices to show that $|\rho(G)|\le 6$. By \cite{Hupp, ATLAS} we easily deduce that $\pi(G/N)=\pi(S)$. Let $\mathfrak{B}=\rho(G)\setminus \pi(S)$. Suppose on the contrary that $|\rho(G)|\geq 7$. It follows that $|\mathfrak B|\geq 4$. By P\'alfy's condition, we must have that $\mathfrak{B}$ spans atleast two edges, say $r\sim l, x\sim y$, where $\{r,l,x,y\}\subseteq \mathfrak B$. We can show that the two edges assumed are disjoint. It is easy to see that if $r\sim l, l\sim x$ are the edges, then $\deg(l)\geq 5$ or $\DG$ contains a $K_5$ by Lemma~\ref{TVtri}.  It follows that there is a $\lambda_1, \lambda_2\in \Irr(N)$ such that $rl|\lambda_1(1)$ and $xy|\lambda_2(1)$.  By Lemma~\ref{TVtri}, each pair $(x,y)$ and $(r,l)$ are contained in some nondisjoint $K_4$'s,  a vertex of degree more than five, or $\DG$ contains subgraph isomorphic to Figure~\ref{fig:2} in which case $S\cong \PSL_2(8)$ or $\Alt_5$. The first two conclusions cannot occur so we may assume that the latter occurs. In fact since $\pi(S)$ should not have an edge, we must have that $G/N\cong S$. Also since no other edge should arise, we must have that $\Delta(N)$ is disconnected with two components each with two vertices or two vertices and a triangle. The former case  cannot occur due to a Theorem by Palfy which states that $N\geq 2^n-1$, where $N$ and $n$ are the number of vertices of the two components.  So the latter occurs. Without loss of generality, assume that $\pi(S)=\{q_1,p_2,q_4\}$ and $\rho(N)=\{p_1,p_3,q_2,q_3,q_1\}$. Let $\varphi\in \Irr(N)$ be such that $q_1q_3|\varphi(1)$. Then by Lemma~\ref{TVtri}, we must have that $\vartheta(1)/\varphi(1)$ is divisible by a prime in $\pi(S)$ for some $\vartheta\in \Irr(G|\varphi)$ or $\varphi$ extends to a $\varphi_0\in \Irr(G)$. Either way we observe that $q_1$ is adjacent to a prime in $\pi(S)$ different from itself, a contradiction.
\end{proof}

\begin{lem}\label{K4}
Assume Hypothesis~\ref{hyp} and let $|\pi(S)|=4$. If the subgraph of $\DG$ on $\pi(M/N)$ is a complete cubic, then $G$ does not exist.
\end{lem}

\begin{proof}
If $\DG$ has 7 vertices then $\DG$ does not contain a $K_4$ and thus there is nothing to prove in this case. Thus we may assume that $|\rho(G)|\ge 8$. It follows that $4\le |\pi(G/N)|\le 5$ by \cite{Hupp, ATLAS, GAP}, which implies that $\mathcal{C}=\rho(G)-\pi(G/N)$ contains at least three primes. Since $\mathcal{C}\subseteq \rho(N)$ and $N$ is solvable, we may use P\'{a}lfy's condition which provides that $\mathcal{C}$ must span at least an edge, say $r\sim l$. Let $\theta(1)\in \cd(N)$ be such that $rl|\theta(1)$. By Lemma~\ref{TVtri} we must have that $\psi(1)/\theta(1)$ divides two distinct primes in $\pi(M/N)$ for some $\psi\in\Irr(M|\theta)$ or $\theta$ extends to $M$ and $M/N\cong \Alt_5$ or $\PSL_2(8)$. It follows immediately that the former must occur. This implies that $\{r,l,p,q\}$ forms a complete cubic for some $q,p\in \pi(M/N)$. This implies that $\deg(q), \deg(p)\ge 5$, a contradiction.
\end{proof}

For the case when $ |\rho(G)|\ge 7$, we will not consider simple groups whose prime graphs have three vertices [Lemma~\ref{simple3primes}] or are complete with at least 4 vertices by Lemma~\ref{K4}.

\begin{lem}\label{Sp}
Assume Hypothesis \ref{hyp}. Then $|\rho(G)|\le 7$ when $S\cong J_1$ and $|\rho(G)|\le 6$ when $S\cong M_{11}$.
\end{lem}

\begin{proof}
By \cite{ATLAS},  we deduce that $G/N=S$.  We use the arguments in the proof of Lemma~\ref{lem:219} above. Let $S\cong M_{11}$. By \cite{ATLAS}, we observe that in $\Delta(S)$, $$\deg(2)=\deg(11)=2, \deg(5)=3 \text{ and }\deg(3)=1. $$ The maximal subgoups of $M_{11}$
 are divisible by either 2 and 3 or 11. It is not difficulty to see that every prime in $\rho(G)\setminus \pi(S)$ is adjacent to either 2 and 3 or to 11. Therefore we can have at most one vertex neighboring 2 and 3 which will make $\deg(2)=4$ since $2\not\sim 3$ in $\Delta(S)$. Also there can only be atmost two neighbours of 11 outside of $\pi(S)$. Thus there should be at most three vertices in $\rho(G)\setminus \pi(S)$. 
 
 Now let $\mathcal{C}=\rho(G)-\pi(S)$. Suppose that $|\mathcal{C}|=3$. Since $\mathcal{C}\subseteq \rho(N)$ and $N$ is simple, we must have that $\mathcal{C}$ spans at least an edge. Let $r,l\in \mathcal{C}$ be such that $r\sim l$. Suppose that $\alpha\in \Irr(N)$ be such that $rl|\alpha(1)$. By Lemma~\ref{TVtri}, $\beta(1)/\alpha(1)$ is divisible by 2 disntinct primes of $\pi(S)$ for some $\beta\in \Irr(G|\alpha)$. This implies that $\beta(1)$ is divisible by 4 distinct primes implying that $\DG$ contains at least eight vertices. Contradicts the previous paragraph. Thus $|\rho(G)|\le 6$. 
 
Now Let $S\cong J_1$. By \cite{ATLAS}, we observe that $$\deg(2)=4, \deg(7)=\deg(19)=3, \deg(3)=\deg(5)=\deg(11)=2. $$Let $\rho(G)-\pi(S)=\mathcal{C}$. Then by Lemma~\ref{lem:219}, each prime in $\mathcal{C}$ is adjacent to either 2 or 19. Since $\deg(2)=4$, two cannot have any more neighbor and thus any prime in $\mathcal{C}$ must be adjacent to 19. However, 19 can only have one more neighbour since $\deg(19)=3$.  
\end{proof}

\begin{lem}\label{lem:others}
Assume Hypothesis \ref{hyp}. Let $S\not\cong \PSL_2(q)$, $q$ a power of some prime number $p$. Then $G$ does not exist.
\end{lem}

\begin{proof}
Assume that $M/N\cong J_1$. Then it follows that $G/N=M/N$ and thus we have that $\pi(G/N)=\pi(M/N)$ with $|\pi(M/N)|=6$. Since $|\rho(G)|\ge 7$, the set $\mathcal{C}=\rho(G)\setminus \pi(G/N)$ is nonempty. Let $x\in \mathcal{C}$. Let $\theta(1)\in\cd(N)$ be such that $x|\theta(1)$. Suppose that $\theta$ is $G$-invariant. Since the Schur multiplier of $J_1$ is trivial, it follows that $\theta$ extends to $G$ by Theorem~\ref{TSchur}. This implies that $\deg(x)=6$ in $\DG$, a contradiction. Hence, $T=I_G(\theta)<G$. Thus $T/N$ is a subgroup of some maximal subgroup $K/N$ of $J_1$. By \cite{ATLAS}, we obtain that $|G:K|$ is divisible by at least 3 primes in $\pi(M/N)$. Thus $\DG$ must contain a $K_4$ as a subgraph. This implies that $|\rho(G)|\geq 8$. This contradicts Lemma~\ref{Sp}.

Suppose that $M/N\cong {}^2B_2(8)$. Then we have that $M/N\cong G/N$ since otherwise by $G/N\cong {}^2B_2(8)\cdot 3$. By \cite{ATLAS}, we have that $\{3,5,7,13\}$ form a complete subgraph of $\DG$. Implying that $\DG$ has at least eight vertices. Thus $|\mathcal{C}|\geq 3$ and by P\'{a}lfys condition, there is $r,l\in \mathcal{C}$ such that $r\sim l$. Let $\theta\in \Irr(N)$ be such that $rl|\theta(1)$. By Lemma~\ref{TVtri}, $\theta$ extends to $M$ or $\psi(1)/\theta(1)$ is divisible by two distinct primes in $\pi(M/N)$. This would result to one of the primes in $\{5,7,13\}$ having degree at least 5, a contradiction. Let $G/N\cong M/N\cong {}^2B_2(8)$. Then again we have that $|\mathcal{C}|\geq 3$. By P\'{a}lfys condition there is $r,l\in \mathcal{C}$ such that $r\sim l$. Let $\theta\in \Irr(N)$ be such that $rl|\theta(1)$ and let $T=I_G(\theta)$. Suppose that $\theta$ is not $G$-invariant. Then $|G:K|$ divides $|G:T|$ for some maximal subgroup $K/N$ of ${}^2B_2(8)$. By \cite{ATLAS}, the possibilities of $\pi(G:K)$ are $\{\{5,13\}, \{2,5,7\},\{2,7,13\},\{2,5,13\}\}$. By Clifford's correspondence theorem, $|G:K|\theta(1)$ divides some degree in $\cd(G|\theta)$. We must have that $|G:K|=\{5,13\}$ since otherwise we would have a $K_5$. Now let $t\neq r,l$ be in $\mathcal{C}$ and let $\phi\in \Irr(N)$ be such that $t|\phi(1)$. Let $G_\phi$ be the stabilizer of $\phi$ in $G$. Then we must have that $G_\phi=G$ since otherwise we have that either deg(5) or deg(13)
is greater than or equal to 5. Also, $\phi$ does not extend to $G$, and thus we may use the projective degrees as in \cite{ATLAS}. In this case we have that $\cd(G|\phi)=\{\phi(1)a|a\in \{40,56,64,104\}\}$. Since $40\phi(1)$ is in $\cd(G|\phi)$, we must have that $\deg(5)\geq 5$ in $\DG$, a contradiction. Thus we must have that $T=G$. If $\theta$ extends to $G$ we have that $\DG$ contains a $K_5$ as a subgraph. Thus we must have that $\theta$ does not extend to $G$. Hence we use the projective degrees provided in  \cite{ATLAS}, we have that $\cd(G|\theta)=\{\theta(1)a|a\in \{40,56,64,104\}\}$. Since $104\theta(1)\in \cd(G|\theta)$ we have that $ \{2, 13,r,l\}$ form a complete cubic which implies that $\deg(13)\geq 5$, a contradiction.

Let   $M/N\cong {}^2B_2(32)$. Then $\pi(G/N)=\pi(M/N)$ and thus $|\mathcal{C}|\ge 3$. By previous arguments, let $rl|\theta(1)$ for some $r,l\in \mathcal{C}$ and $\theta\in \Irr(N)$. By Lemma~\ref{TVtri}, %$\theta$ extends to $M$ or 
$\chi(1)/\theta(1) $ is divisible by two distinct primes in $\pi(M/N)$ for some $\chi\in\Irr(M|\theta)$. %If the former occurs then $\DG$ contains a $K_5$, a contradiction. Thus we may assume that the latter occurs and thus 
We now have that $\DG$ contains a complete cubic. This implies that $|\rho(G)|\ge 8$, which implies that $|\mathcal{C}|\ge4$. By P\'{a}lfy's condition, $\mathcal{C}$ spans at least two edges. Thus it suffices to assume that there is $\{x,y\}$ different from $\{r,l\}$ such that $x \sim y$. By previous arguments we obtain that at least one vertex in $\pi(M/N)$ has degree at least 5, a contradiction. A similar argument applies to any $M/N\in \{\PSL_4(q), \PSU_3(q^2) \}$ where $q$ { a power of a prime }$p$ with $|\pi(M/N)|=4$ such that $\pi(M/N)$ does not span a complete cubic.

%Suppose that $M/N\cong \Alt_8$. Clearly $\pi(G/N)=\pi(M/N)$ and it follows that $\Delta(G/N)$ is a $K_4$ if $M/N<G/N$. This cannot occur by Lemma~\ref{K4}. Thus $G/N\cong \Alt_8$ and $|\mathcal{C}|\geq 3$. Let $r\sim l$ be in $\mathcal{C}$ and let $\phi\in \Irr(N)$ be such that $rl|\phi(1)$. Suppose that $T=I_G(\phi)=G$. Then evidently $\phi$ does not extend to $G$ and thus using the projective degrees in \cite{ATLAS}, we see that $\cd(G|\phi)=\{\phi(1)a|a\in \{8,24,56,64\}\}$. This implies that $r$ and $l$ are adjacent to 7 since $\phi(1)56\in \cd(G|\phi)$. This implies that $\deg(7)\geq 5$, a contradiction. Therefore we may assume that $T<G$. In this case by considering maximal subgroups of $\Alt_8$, we have that $\{r,l,5\}$ or $\{r,l,7\}$ is a triangle which implies makes the degree of 5 or 7 greater than or equal to 5 or $T/N\le \Alt_7$. XXXXXXXX

Suppose that $\pi(M/N)|\geq 5$ and $\Delta(M/N)$ is not complete. Then $M/N\cong \PSL_3(q)$ or $\PSU_3(q^2)$ with $q$ a power of $p$ and $q-1\neq 2^i3^j, i\ge 1, j\ge 0$ or $q+1\neq 2^i3^j$, $i,j\ge 0$ respectively. In this case $\DG$ contains a subgraph isomorphic to Figure~\ref{fig:5}. This implies that $|\rho(G)|=9$ and $\DG$ is isomorphic to Figure~\ref{fig:4}(a) and $\mathcal{C}$ is contained in one of the complete cubic subgraphs while four primes of $\pi(M/N)$ form the other complete subgraph.  But then by using Lemma~\ref{TVtri}, we should have that two primes in $\mathcal{C}$ form a complete subgraph with two primes in $\pi(M/N)$, this is not achieved, a contradiction.
\end{proof}

The following results or facts  will be used to prove subsequent results.

\begin{rem}\textup{\cite[Remark 2]{Zainab}}\label{remark2}
Let $M/N\cong \PSL_2 (q)$ where $q\not= 9$ and $\theta \in\Irr(N)$. If $\theta$ is $M$-invariant, then there are irreducible characters in $\Irr(M|\theta)$ such that their degrees are divisible by $\theta(1)(q \pm 1)$. % In fact we can conclude the result, as $M (PSL 2 (q)) is cyclic of order at most two.
\end{rem}

%\begin{lem}\textup{\cite[Lemma 3.4]{Zainab}}\label{Zainab}
%Assume Hypothesis \ref{hyp}. Let  $S\cong  \PSL_2 (q)$, where $q$ is a power of $p$
%and $|\pi(M/N)|\ge 4$. Let $\theta \in \Irr(N)$ and set $n = |\pi(\theta(1)) \cap (\rho(N) \setminus \pi(G/N))|$. Assume
%$n\ge 1$, when $|\pi(G/N)| \ge 8$ and $n \ge 2$, when $|\pi(G/N)|\le 7$. If $\DG$ is $K_5$-free, then
%$|\rho(G/N)| \le 7$ and the following hold:\begin{enumerate}
%\item[(i)] If $\theta$ is not $M$-invariant, then there exists $\epsilon \in \{1, -1\}$, such that $|\pi(q + \epsilon)| = 1$
%and $\pi(p(q + \epsilon)\theta(1))$ forms a complete subgraph in $\DG$.
%\item[(ii)] If $\theta$ is $M$-invariant, then $|\pi(q \pm 1)| \le 2$ and $\pi(G/N) = \pi(M/N)$. In particular
%$|\pi(G/N)| \le 5$.\end{enumerate}
%\end{lem}

\begin{lem}\textup{\cite[Lemma 3.7]{McVey}}\label{McVey}
Let $G$ be a group  and let $N\unlhd G$ be such that $G/N\cong \PSL_2(q), q=p^f$
for some prime  $p$ and some positive integer $f$. Let $\theta\in \Irr(N)$ and $T$ be the stabilizer of $\theta$ in $G$ with $|G : T | = t$. Suppose that $T/N\cong \PSL_2(p^k)$ or $\PGL_2(p^k)$ with $7\le p^k<q=p^f$.
Then $\pi(t)$ has nontrivial intersection with each of the three sets ${p}, \pi(q - 1)$, and $\pi(q + 1)$
where $q > 5$. Moreover, at least one of the following holds:\begin{enumerate}
\item[(a)] There are degrees $\psi_1(1) , \psi_2(1)\in \cd(G|\theta)$ so that $p(q - 1)$ divides $\psi_1(1)$ and $p(q + 1)$ divides $\psi_2(1)$.
\item[(b)] $T /N$ is isomorphic to one of $\PSL_2 (9)$ or $\PGL_2 (9)$ and both $6t\theta(1)$ and $15t\theta(1)$ divide
degrees in $\cd(G|\theta)$.
\end{enumerate}

\end{lem}

\begin{lem}\textup{\cite[Lemmas 3.2-3.7]{McVey}}\label{conclusion} Assume Hypothesis \ref{hyp}. Let $S\cong \PSL_2(2^f)$, $f\ge 3 $. Let $\theta\in\Irr(N)$ and $T=I_M(\theta)$. Then  
\begin{enumerate}
\item[(A)] $|M:T|$ is divisible by all primes in two of the three sets $\{2\}, \pi(2^f - 1)$, and $\pi(2^f + 1)$.
\item[(B)] $|M:T|$ is divisible by all primes in $\rho(M/N) \setminus \{2, 3\}$ and some degree in $\cd(M|\alpha)$ is divisible by one of $2|M:T|\alpha(1)$ or $3|M:T|\alpha(1)$.
\item[(C)] Lemma~\ref{McVey} applies.
\item[(D)] $|M:T|$ is divisible by all primes in $\pi(M/N) \setminus \{2, 3, 5\}$, the integer $q > 5$, and either $3|M:T|\theta(1), 4|M:T|\theta(1)$, and $5|M:T|\theta(1)$ divide degrees in $\cd(G|\theta)$, or $6|M:T|\theta(1)$ divides a degree in $cd(G|\theta)$.
\end{enumerate}

\end{lem}

\begin{lem}\label{psl2q}
Assume Hypothesis \ref{hyp}. Let $S\cong \PSL_2(q)$ where $q$ is a power of some prime $p\neq 2$. Then $G$ does not exist. 
\end{lem}

\begin{proof}
By Lemma~\ref{simple3primes}, we have that $|\pi(S)|\ge 4$.  

{\bf Case 1: $|\pi(S)|\ge 6$}

If $|\pi(S)|\ge 7$, then by \cite[Theorem 3.1]{White2} we have that $|\pi((q-1)(q+1))|=6$ with (in no particular order)$$\{|\pi(q-1)|,|\pi(q+1)|\}=\{3,4\} \text{ or }\{2,5\}.$$In both cases we have that $\deg(2)=5$ in $\Delta(S)$. So it suffices to consider that case when $|\pi(S)|=6$. 
In this case we must have that $|\pi((q-1)(q+1))|=5$ with $$\{|\pi(q-1)|,|\pi(q+1)|\}=\{3,3\} \text{ or }\{2,4\}.$$ We claim that $|\pi(G/N)|=|\pi(S)|$. Suppose that $t\in \phantomsection(G/N)\setminus \pi(S)$. By \cite[Theorem A]{WhiteE}, we have that $t\sim r$ for all $r\in \pi(q^2-1)$. But $\deg(t)$ becomes five, a contradiction. Thus $\mathcal{C}=\rho(G)-\pi(G/N) $  is nonempty. 
Let $r\in \mathcal{C}$ and let $\theta\in \Irr(N)$ be such that $r|\theta(1)$. Suppose that $\theta$ is $M$-invariant, then by Remark~\ref{remark2} we have that both $\{r\}\cup\pi(q\pm1)$ form complete subgraphs of $\DG$. In either case we obtain that $\deg(2)\ge 5$, a contradiction. Thus we may assume that $M_\theta=I_M(\theta)<M$. In this case we have that $|M:M_\theta|\theta(1)$ divides the degrees of all members of $\Irr(M|\theta)$. We must have that $M_\theta/N\le K/N $ for some maximal subgroup $K/N$ of $S$. This therefore implies that $|M:K|$ divides $ |M:M_\theta|$ and thus $|M:K|\theta(1)$ divides the degrees of all irreducible characters in $\Irr(M|\theta)$. By \cite[Hauptsatz II.8.27]{Hp}, the indices
of all maximal subgroups of $M/N$ is divisible by at least three distinct primes. In this case we obtain that some degree in $\pi(q\pm 1)$ has degree at least 5, a contradiction.

{\bf Case 2: $|\pi(S)|=5$}

By \cite[Theorem B]{TVtri}, we have that $\Delta(S)$ contains a triangle. We claim that $2\le |\pi(q\pm 1)|\le 3$. Suppose on the contrary that $\Delta(S)$ contains a $K_4$ and an isolated vertex. In this case we have that $|\rho(G)|\ge 8$. Also, we can see that if $t\in \pi(G/N)-\pi(S)$, then $\Delta(G/N)$ contains a $K_5$. So we must have that $|\pi(S)|=|\pi(G/N)|$. Now we must have that $|\mathcal{C}|\ge 3$ and by P\'{a}lfy's condition, $\mathcal{C}$ spans atleast an edge. Let $r,l\in \mathcal{C}$ be such that $rl|\alpha(1)$ for some $\alpha\in \Irr(N)$. Then by Lemma~\ref{TVtri}, we obtain that $\DG$ contains two $K_4$'s whose intersection is nonempty, a contradiction. So we may assume that $2\le |\pi(q\pm 1)|\le 3$. Suppose that $|\pi(G/N)|>|\pi(S)|$ and let $t\in \pi(G/N)\setminus \pi(S)$. Again, we have that $t(q\pm 1)$ divides some degree of $\cd(G/N)$. This implies that $\Delta(G/N)$ would contain a subgraph isomophic to Figure~\ref{fig:5}. This implies that $|\rho(G)|=9$. Which implies that $|\mathcal{C}|\ge 3$ and suppose that $r\sim l$ in $\Delta(N)$ with $r,l\in \mathcal{C}$. Let $\theta\in\Irr(N)$ be such that $rl|\theta(1)$. By Lemma~\ref{TVtri}, $\{r,l,u,v\}$ form a complete cubic for some primes $u, v\in \pi(S)$. In this case we obtain that $\DG$ must contain two complete cubic subgraphs whose intersection is non-empty, a contradiction. Thus $|\pi(S)|=|\pi(G/N)|$. Let $r,l\in \mathcal{C}$ be such that $r\not\sim l$. Let $\theta_1,\theta_2\in \Irr(N)$ be such that $r|\theta_1(1)$ and $l|\theta_2(1)$. Suppose that $\theta_i$ is $M$-invariant for some $i\in \{1,2\}$. Then it follows by Remark~\ref{remark2} that $\theta_i(1)(q\pm 1)$ divide some degrees of irreducible charactes in $\Irr(M|\theta_i)$. This implies that $\DG$ admits a subgraph isomorphic to Figure~\ref{fig:5}, which in turn implies that $|\rho(G)|= 9$. In particular, $\DG$ is isomorphic to Figure~\ref{fig:4}(a). In this case we must have that $|\mathcal{C}|=4$ and by P\'{a}lfy's condition, we have that $\mathcal{C}$ spans at least two edges. This inturn (by Lemma~\ref{TVtri}) implies that $\DG$ has two complete cubic subgraphs and $r,l$ are not contained in the same $K_4$. It will suffice to let $x,y\in \mathcal{C}$ with $\{r,l\}\cap \{x,y\}=\emptyset$ and let $\alpha, \beta\in \Irr(N)$ be such that $xr|\alpha(1)$ and $ly|\beta(1)$. Then it is not hard to see that $I_M(\alpha), I_M(\beta)<M$ by Remark~\ref{remark2}. We see that $I_M(\alpha)/N$ and  $I_M(\beta)/N$ are some subgroups of $\PSL_2(q)$ listed in \cite[Hauptsatz II.8.27]{Hp}. Observe that their indices in $M/N$ will be divisible by either $p$ or 2. It is then not hard to see that the two $K_4$'s intersection is non-empty or $\deg(2)\ge 5$, a contradiction. 
% % % % % % % % % % % % % % % % % % % % % % % % % % % % % % % % % % % % % %
%
%%
%%
%%By the  Suppose without loss of generality that $i=1$, then we would have that $\deg(r)\ge 6$, a contradiction so we must have that $\theta_i$ are not $M$-invariant for each $i$. 
It follows that both $\theta_1, \theta_2$ are not $M$-invariant.
Let $M_{\theta_i}$ be the stabilizer of the character $\theta_i$ in $M$ for each $i\in\{1,2\}$. We observe that $\overline{H}=M_{\theta_1}/N$ is one of subgroups of $\PSL(2,q)$ listed in \cite[Hauptsatz II.8.27]{Hp}. Suppose that $\overline{H}$ is the elementary abelian $p$-group. Then we have that $\{r\}\cup \pi(q^2-1)$ form a $K_5$, a contradiction. Now suppose that $\overline{H}$ is the cyclic group of order $z$ where $z\left|\frac{q\pm 1}{(q-1,2)}\right. $. In this case we have that $\{r,p\}\cup \pi(q\pm 1)$ forms a complete subgraph. This obtains that $\deg(2)\ge 5$, contradiction. Now suppose that $\overline{H}\cong D_{z}$ with $z$ as above. Let $C/N\cong C_z\le D_z$. Then $\theta_1$ extends to $C$ by \cite[Corollary 11.22]{Isaacs}. Let $\hat{\theta_1}$ be an extension of $\theta_1$ in $\Irr(C)$. If $\hat{\theta_1}$ is $M_{\theta_1}$ invariant, then $\hat{\theta_1}$ extends to $M_{\theta_1}$ and thus $2\theta_1\in \cd(M_{\theta_1}|\theta_1)$ by Gallagher's Theorem. Otherwise $I_{M_{\theta_1}}(\hat{\theta_1})=C$ and by Clifford's theorem, we have that $2\theta_1(1)\in \cd(M_{\theta_1}|\theta_1)$. This therefore implies that $2|M:M_{\theta_1}|\theta_1(1)$ is a degree in $\cd(M)$. Since $p$ is a divisor of $|M:M_{\theta_1}|$ we must have that $\{r,p,2\} $ form a triangle. This would imply that  $\deg(2)\ge 5$, a contradiction. Suppose that $\overline{H}\cong \Alt_4$. 

{\bf claim:} {\em $p\neq 3$ and $2^2$ is the highest power of $2$ that divides $|S|$}. In particular, we claim that $\Alt_4$ is a Hall \{2,3\}-subgroup of $S$. Suppose that $p=3$. Then $f> 3$ and thus 3 divides $|M:M_{\theta_1}|$. Since $|M:M_{\theta_1}|\theta_1(1)$ divides some degree in $\cd(M|\theta_1)$, we have that $|\pi(S)|=\pi(M:M_{\theta_1})$ or $2 \notin \pi(M:M_{\theta_1})$. If $2\in \pi(M:M_{\theta_1})$, then $\DG$ contains a $K_6$ as a subgraph, a contradiction. Also, since 3 divides $|M:M_{\theta_1}|$, then we have that $|\pi(M:M_{\theta_1})|=4$. In this case we also have that $\DG$ must contain a $K_5$ with vertices $\{r\}\cup \pi(M:M_{\theta_1})$. 

Suppose that $\theta_1$ extends to $M_\theta$, then by Gallagher's theorem, $2\theta_1(1)$ or $3\theta_1(1)$ divides some degree in $\cd(M_{\theta_1}|\theta_1)$, in which case we obtain that $\DG$ contains a $K_5$ as a subgraph. So we may suppose that $\theta_1$ does not extend to $M_{\theta_1}$. By \cite[Corollary 11.29]{Isaacs}, there is a  $\psi\in \Irr(M_{\theta_1}|\theta_1)$ which is divisible by either 2 or 3. This as well leads to $\DG$ admitting $K_5$ as a subgraph.  

Suppose that $\overline{H}$ is a semidirect product of an elementary abelian group of order $p^m$ with a cyclic group of order $t$, where $t|(p^m-1)$ and $t|(p^f-1)$. $\overline{H} $  is a Frobenius group with Frobenius Kernel $K/N$, an elementary abelian $p$-group. This implies that 2 divides $|M:M_{\theta_1}|$. If $\theta_1$ does not extend to $K$, then the degrees in $\cd(M|\theta_1)$ are divisible by $p$. Thus we must have that $\{p,2,r\}$ form a tringle, in which case $\deg(2)\ge 5$, contradiction. Thus we may assume that $\theta_1$ does not extend to $K$. Let $Q/N$ be the Sylow q-subgroup for some prime $q|t$. Then $Q/N$ is cyclic and so $\theta_1$ extends to $Q$. By \cite[Corollary 11.31]{Isaacs}, we have $\theta_1$ extends to $M_{\theta_1}$. By Gallagher's theorem, $|M_{\theta_1}:K|\theta_1(1)\in \cd(M_{\theta_1}|\theta_1)$ and by Clifford's theorem we must have $|M:M_{\theta_1}||M_{\theta_1}:K|\theta_1(1)\in \cd(M|\theta_1)$. This implies that $\{r\}\cup\pi(p^{2f}-1)$ forms a complete $K_5$, a contradiction.

Suppose that $M_\theta/N\cong \Alt_5$. We claim that $p\in \pi(M:M_{\theta_1})$. Suppose the contrary that $p\notin \pi(M:M_{\theta_1})$. Then it implies that $p\in \pi(\Alt_5)\setminus\{2\}$. In particular, $p=3$ or $5$ Suppose that $p=3$, then since $\PSL_2(3)$ is solvable. Then it is easy to see that $f\neq 1$ in this case. Also, $\pi(\PSL_2(5))$ contains only 3 primes and so in this case $f\neq 1$ as well. Thus $p\in \pi(M:M_{\theta})$.  Now, by the projective degrees in \cite{ATLAS}, we see that $2|M:M_{\theta_i}|\theta_1(1)$ divides some degree in $\cd(M|\theta_1)$. This implies that $\{r,p,2\}$ forms a triangle in which case $\deg(2)\ge 5$, a contradiction.  

Finally, suppose that $\overline{H}\cong \PSL_2(p^m)$ or $\PGL_2(p^m)$. By Lemma~\ref{McVey}, we must have that $\pi(M:M_{\theta_1})$ has nontrivial intersection with each of the three sets $\{p\}, \pi(q\pm 1)$ and $|\pi(M:M_{\theta_1})|p(q-1)\theta(1)$ and $|\pi(M:M_{\theta_1})|p(q+1)\theta(1)$ divides some degrees in $\cd(M|\theta)$ or $6|M:M_\theta|\theta(1)$ divides some degree in $\cd(M|\theta)$. In both cases $\deg(2)$ will exceed 5.

% So we may assume that 2 does not divide $|M:M_{\theta_1}|$. In this case we have that $|M:M_{\theta_1}|\geq 3$. Thus $\{r\}\cup \pi(M:M_{\theta_1})$ is a complete cubic. This results to at least one prime in $\pi(M:M_{\theta_1})$ being adjacent to at least 5 vertices, a contradiction.

{\bf Case 3: $|\pi(S)|=4$}.

{\bf Subcase 1: $|\pi(q+\epsilon)|=1$ and $|\pi(q-\epsilon)|=3, \epsilon\in \{-1,1\}$}

 We claim that $|\pi(S)|=|\pi(G/N)|$. Suppose on the contrary that $t\in \pi(G/N)\setminus \pi(S)$. It follows that  both $t\sim w$ for each $w\in (q^2- 1)$. This results into a $K_4$. Thus we have that $\DG$ has at least eight vertices. This implies that $\mathcal{C}$ contains not less than three primes. Thus we can find $r,l\in\mathcal{C}$ such that $r\sim l$ in $\Delta(N)$. Let $\theta\in\Irr(N)$ be such that $rl|\theta(1)$. Then by Lemma~\ref{TVtri}, $r, l$ together with two primes in $\pi(S)$  form a $K_4$. So we have two $K_4$'s whose intersection in not empty, a contradiction.  Now we have that $|\pi(S)|=|\pi(G/N)|$. So that $\mathcal{C}$ contains at least three primes and by P\'{a}lfy's condition, there is $x,y\in \mathcal{C}$ such that $x\sim y$. Let $\alpha\in \Irr(N)$ be such that $xy|\alpha(1)$ and observe that by Lemma~\ref{TVtri} $\{x,y,u,v\}$ form a complete subgraph for some primes $u,v\in \pi(S)$. It follows immediately that $|\rho(G)|\ge 8$ and thus $\mathcal{C}$ contains at least four primes and thus must contain atleast two edges. It suffices to assume that $r, l\in \mathcal{C}$ are distinct from $x$ and $y$ and $r\sim l$. Letting $\gamma\in \Irr(N)$. It is not hard to see that $r,l$ will be contained in a $K_4$ and the intersection of two $K_4$'s is nonempty or $p$ is contained in one of the $K_4$'s and so there will be a prime $w\in \pi(S)$ with $\deg(w)\ge 5$, a contradiction.

{\bf Subcase 2: $|\pi(q\pm 1)|=2$}

We claim that $|\pi(S)|=|\pi(G/N)|$. Suppose on the contrary that $t\in \pi(G/N)\setminus \pi(S)$. It follows that  both $t\sim w$ for each $w\in (q^2- 1)$ form complete triangles which results into $\deg(2)=3$. Since $\mathcal{C}$ contains at least two primes, let $r, l\in \mathcal{C}\subseteq \rho(N)$. If we let $\theta_1, \theta_2\in \Irr(N)$ be such that $r|\theta_1(1)$ and $l|\theta_2(1)$, then whether $\theta_1$ and $\theta_2$ extend to $M$ or not, we obtain that $r$ and $l$ are both adjacent to 2 which implies that $\deg(2)\ge 5$, a contradiction.

Since $|\pi(G/N)|=|\pi(S)|=4$, we have that $|\mathcal{C}|\ge 3$. Let $r,l\in \mathcal{C}$ be such that $r\sim l$ and let $\theta\in \Irr(N)$ be such that $rl|\theta(1)$. If $\theta$ is $M$-invariant, then by Remark~\ref{remark2}, we must have that $\theta(1)(q\pm 1)$ divide some degrees in $\cd(M|\theta)$. This results into a subgraph with 5 vertices and 3 complete vertices. This does not occur as a subgraph of any graph in consideration, a contradiction. We must thus assume that $\theta$ is not $M$-invariant. Again, we must have that $M_\theta/N=I_M(\theta)/N$ is one of the subgroups discussed above. If $M_\theta/N$ is one of the abelian subgroups of $S$, then $|M:M_\theta|$ is divisible by at least three primes of $\pi(S)$. This implies that $\DG$ contains a $K_5$. So we must have that $M_\theta/N$ is nonabelian. 

Let $M_\theta/N$ be a Frobenius group. We obtained that $|M:K|\theta(1)\in \cd(M|\theta)$ with $|M:K|$ divisible by all primes in $\pi(q^2-1)$. This also implies that $\DG$ will contain a $K_5$, a contradiction. 

Let $M_\theta/N\cong \Alt_4$ or $\Sym_4$.  This implies that $\pi(M:M_\theta)\ge 2$. By previous argument we obtain that $2|M:M_\theta| \theta(1)$ or $3|M:M_\theta|\theta(1)$ divides some degrees in $\cd(M|\theta)$ and hence some degree in $\cd(G|\theta)$. This will imply that $\DG$ contains a $K_5$.

Suppose that $M_\theta/N\cong \Alt_5$. By the projective degrees in \cite{ATLAS}, we must have that $2\theta(1)$ divides some degree in $\cd(M_\theta|\theta)$, whether $\theta$ extends to $M_\theta$ or not. By arguments in \cite[Proof of Theorem 3.3]{Lewis4}, we obtain that $2p\theta(1)$ divide some degree in $\cd(M|\theta)$. This however implies that $\deg(2)\ge 5$, a contradiction.
 
Finally, if $M_\theta/N\cong \PSL_2(2,p^m)$ or $\PGL_2(p^m)$. By Lemma~\ref{McVey}, we must have that $|M:M_\theta|$ is divisible by at least three primes in $\pi(S)$. In this case, we obtain that $\{r,l\}\cup \pi(M:M_\theta)$ form a $K_5$, a contradiction.
\end{proof}

\begin{lem}\label{4primes}
Assume Hypothesis \ref{hyp}. Let $S\cong \PSL_2(2^f)$ be such that $|\pi(S)|=4$. Then $G$ does not exist.
\end{lem}

\begin{proof} We consider cases when $|\pi(G/N)|= |\pi(S)|$ and $|\pi(G/N)|\not= |\pi(S)|$ differently. 

{\bf Case 1: $|\pi(S)|= |\pi(G/N)|$}

We have that $|\mathcal{C}|\ge 3$. Since $\mathcal{C}\subseteq \cd(N)$ and $N$ is solvable, we can use P\'{a}lfy's condition and assume that there is a $r,l\in \mathcal{C}$ such that $r\sim l$ in $\Delta(N)$. Let $\theta\in \Irr(N)$ be such that $rl|\theta(1)$. If $\theta$ is $G$-invariant, then since the Schur multiplier of $S$ is trivial, we have that $\theta$ extends to $G$. Therefore, by Gallagher's Theorem, $\theta(1)(2^f\pm 1), 2^f\theta(1)\in \cd(G)$. In this case, we obtain $\deg(r)$ and $ \deg(l)$ is at least 5, a contradiction. Thus we may assume that $\theta$ is not $G$-invariant. 

Let $T=I_M(\theta)$. Then $T/N$ is one of the Dickson's list of subgroups of $\PSL_2(q)$ in \cite[Hauptsatz II.8.27]{Hp}. Suppose that $T/N$ is an elementary abelian 2-group, then we have that $|\pi(M:T)|\ge 3$, which would result in a $K_5$ as a subgraph of $\DG$. Now, suppose that $T/N$ is a cyclic group or a Frobenius group. Then by Lemma~\ref{conclusion}, $|M:T|$ is divisible by all primes in two of the sets $\{2\}, \pi(2^f-1)$, and $\pi(2^f+1)$. So we may assume that $|G:T|$ is divisible $\{2\}$ and $\pi(2^f+\epsilon), \epsilon\in \{1,-1\}$ in which $|\pi(2^f+\epsilon)|=1$. In this case we obtain that $\DG$ contains a $K_4$ and thus $|\rho(G)|\ge 8$. This also implies that  $|\mathcal{C}|\geq 4$ and hence must span at least 2 edges by P\'{a}lfy's condition. Let $u,v\in \mathcal{C}$ be different from $r, l$. and let $\alpha\in \Irr(N)$ be such that $uv|\alpha(1)$. 
Let $T=I_M(\alpha)=M$. Then we must have that $\alpha$ extends to $M$, in which case we obtain that $\DG$ contains two $K_4$'s whose intersection is nonempty.
We may thus assume that $T<G$. 
Then by Lemma~\ref{conclusion}, we observe  that in all the cases, we must have some prime in $\pi(S)$ with at least degree 5.

{\bf Case 2: $|\pi(S)|<|\pi(G/N)|$}

By \cite{Hupp}, we have that $f=5,7$. In particular $2^f-1$ is a Merssene prime $s$ and $2^f+1$ is a product of powers of two primes $x$ and $y$.  By \cite{ATLAS}, we observe that $\deg(x)=\deg(y)=2$ and $\deg(s)=1$  in $\Delta(G/N)$. We have that $|\mathcal{C}|\ge 2$. 

{\bf Claim: $|\mathcal{C}|\le 2$. } %\em in particular, $\mathcal{C}$ does not span an edge}

To show this, suppose that $|\mathcal{C}|\ge 3$ and let $r,l,a\in \mathcal{C}$ such that $r\sim l$. Let $\theta,\phi\in \Irr(N)$ be, respectively, irreducible constituents of $\chi_N, \vartheta_N$ where $rl|\chi(1), a|\vartheta(1)$. Then $rl|\theta(1)$ and $a|\phi(1)$. Let $M_\theta$ and $M_\phi$ be the stabilizers of $\theta$ and $\phi$ in $M$ respectively. Suppose that $M_\theta=M$, Then we obtain $\deg(r), \deg(l)\geq 5$, a contradiction.

\begin{figure}[htb!]\centering
\begin{tikzpicture}
	\vertex (w) at (18:1) [label=right:$s$]{};
	\vertex (u) at (90:1) [label=above:$2$]{};
	\vertex (v) at (162:1) [label=left:$x$]{};
	\vertex (x) at (234:1) [label=left:$y$]{};
	\vertex (y) at (306:1) [label=right:5/7]{};
	\path 
 		(x) edge (y)
% 		(x) edge (w)
 		(x) edge (v)
% 		(x) edge (u)
 		(y) edge (w)
 		(y) edge (v)
% 		(y) edge (u)
	;
\end{tikzpicture}
\caption{$\Delta(\Aut(\PSL_2(2^f))), f = 5,7$}
\label{fig:Aut}
\end{figure}
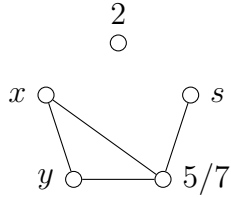 \FloatBarrier
\begin{figure}[htb!]\centering
%\begin{tikzpicture}
%	\vertex (f) at (30:1.2)[label=above:$r$]{};
%	\vertex (e) at (90:1.2)[label=above:$l$]{};
%	\vertex (d) at (150:1.2)[label=above:$s$]{};
%	\vertex (c) at (210:1.2)[label=below:$x$]{};
%	\vertex (b) at (270:1.2)[label=below:$y$]{};
%	\vertex (a) at (330:1.2)[label=below:5/7]{};
%	\path 
%		(b) edge (c)
%		(b) edge (e)
%		(b) edge (f)
%		(e) edge (f)
%		(c) edge (f)
%		(c) edge (e)
%		(c) edge (d)
%		(d) edge (a)
%		(d) edge (e)
%		(a) edge (f)
%		(b) edge (a)
%		;
%\draw (0,-2) node [below] {(a)} circle (0);
%\end{tikzpicture}\hspace{2cm}
\begin{tikzpicture}
	\vertex (w) at (90-360/7:1.4) [label=right:$2$] {};
	\vertex (u) at (90:1.4) [label=above:$l$] {};
	\vertex (v) at (90+360/7:1.4) [label=left:$x$] {};
	\vertex (x) at (90+720/7:1.4) [label=left:$y$] {};
	\vertex (y) at (90+1080/7:1.4) [label=left:5/7] {};
	\vertex (a) at (90+1440/7:1.4) [label=right:$s$] {};
	\vertex (b) at (90+1800/7:1.4) [label=right:$r$] {};
	\path 
 		(w) edge (u)
% 		(u) edge (v)
 		(v) edge (x)
 		(y) edge (a)
 		(b) edge (a)
 		(x) edge (y)
 		(b) edge (w)
% 		(v) edge (w)
 		(u) edge (a)
 		(v) edge (y)
% 		(x) edge (a)
% 		(y) edge (b)
 		(a) edge (w)
 %		(b) edge (u)
	;
%\draw (0,-2) node [below] {(b)} circle (0);	
\end{tikzpicture} 

\caption{Impossible subgraphs}
\label{fig:imp}
\end{figure}
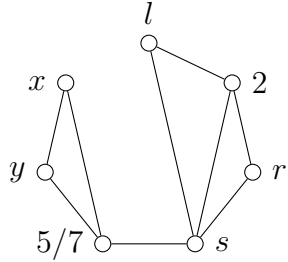\FloatBarrier

Thus we may assume that $M_\theta<M$. Since $|\pi(M:M_\theta)|\ge 3$ implies that $\DG$ contains a $K_5$, we must have that $M_\theta\in \{ D_{2xy},2^f\cdot s\} $  by \cite{Lewis4} and Lemma~\ref{conclusion}.  We may assume that $\pi(M:M_\theta)=\{2,s\}$ since otherwise $|\pi(M:M_\theta)|=3$. This makes $\deg(2)=3$ and $\deg(s)=4$. Now, we consider $M_\phi$. If $M_\phi=M$ then $\deg(s)=5$, a contradiction so we must have that $M_\phi<M$ and by Lemma~\ref{conclusion}, we have $|\pi(M:M_\phi)|=3$ or $\pi(M:M_\phi)=\{2,s\}$ as argued above. Either way we obtain $\deg(b)\ge 5$ for some $b\in \{2,s\}$, a contradiction.

Now we show that $G$ with this property does not exist. Since $|\rho(G)|=7$, then we need to show that $\DG$ does not contain a $K_4$. Let $r,l\in\mathcal{C}$ be such that $r\not\sim l$ and let $\theta, \phi\in \Irr(N)$ be such that $r|\theta(1)$ and $l|\phi(1)$. Let $M_\theta, M_\phi$ be the stabilizers of $\theta, \phi$ in $M$ respectively. Then we observe that $\pi(M:M_\theta)=\pi(M:M_\phi)=\{2,s\}$ by Lemma~\ref{conclusion}, otherwise $\DG$ contains a $K_4$ or an impossible subgraph when one of $\phi, \theta$ is $M$-invariant. In this case we obtain a subgraph isomorphic to Figure~\ref{fig:imp}.

Figure~\ref{fig:imp} is a subgraph of both 4-regular graphs with 7 vertices. Consider Figure~\ref{fig:1}. Observe that the graph is vertex transitive so we need to consider  only one possibility. We may suppose that, (in that order), $$(p_1,p_2,p_3,q_1,q_2,q_3,q_4)=(h, l,s,y,x,r,2), h=5\text{ or }7$$
Observe that $2\sim x$. Let $\chi\in \Irr(G)$ be such that $2x|\chi(1)$. And let $\theta\in \Irr(N)$ be the be such that $[\chi_N,\theta]\neq0$. Since $2\not\sim x$  in $\Delta(G/N)$, it is evident that $\theta(1)\neq 1_N$. Let $M_\theta=I_M(\theta)=M$. Then $\theta$ extends to $M$ and by Gallagher's theorem, we have that $$\chi(1)\in \{\theta(1), 2^f\theta(1), (2^f-1)\theta(1), (2^f+1)\theta(1)  \} $$ 
Observe that if $\chi(1)= \theta(1), \text{ or }2^f\theta(1)$, we have that $x|\theta(1)$ in both cases. But then by $(2^f-1)\theta(1), (2^f+1)\theta(1)$ we must have that $x\sim s$ which is not the case. Thus we may assume that $M_\theta<M$. By Lemma~\ref{conclusion}, we have that $|\pi(M:M_\theta)|\ge 2$ with $$\pi(M:M_\theta)\in \{\{x,y,s\}, \{x,y,2\}, \{s,2\}\}.$$ But $2, x$ do not share any of the above possible neighbors.
%
%Now we consider Figure~\ref{fig:2}. We have two subgraphs isomorphic to Figure~\ref{fig:imp}. $$(p_1,p_2,p_3,q_1,q_2,q_3,q_4)=(x,h,y,l,s,2,r) \text{ or }(s,r,2,h,y,x,l) $$
%We observe that in both cases that $h\sim 2$. But $h\not\sim 2$ in $\Delta(G/N)$. Let $\chi\in \Irr(G)$ be such that $2h|\chi(1)$ and let $\theta\in \Irr(N)$ be the irreducible constituent of $\chi_N$ and observe that $\theta\neq 1_N$. 
%Let $M_\theta=I_M(\theta)=M$. Then $\theta$ extends to $M$ and by Gallagher's theorem, we have that $$\chi(1)\in \{\theta(1), 2^f\theta(1), (2^f-1)\theta(1), (2^f+1)\theta(1)  \} $$ 
%Observe that if $\chi(1)= \theta(1), \text{ or }2^f\theta(1)$, we have that $h|\theta(1)$ in both cases. This implies that $h\in \rho(N)$. But then $\{h,l,r\}$ is a set of three vertices in $\Delta(N)$ which do not span an edge, a contradiction. Thus we may assume that $M_\theta<M$. XXXXXXXXX
\end{proof}

\begin{lem}\label{onAS}
Assume Hypothesis \ref{hyp}. Let $M/N=\PSL_2(2^f)$ for some $f\geq 6$. Then % \begin{itemize}
 if $\pi(G/N)\neq \pi(M/N)$, then $ |\pi(S)|= 5$, and $\Delta(S)$ is not isomorphic to Figure~\ref{fig:p}(a).
%\item[(ii)] if $\Delta(S)$ contains a $K_4$ as a subgraph or $|\pi(S)|\ge 7$, then $G/N=M/N.$ 
%\end{itemize} 
\end{lem}

\begin{proof}
By Lemma~\ref{simple3primes} and Lemma~\ref{4primes}, we have $|\pi(M/N)|\geq 5.$ Since $\pi(G/N)\neq \pi(M/N)$, it follows that $G/N=M/N\langle \alpha\rangle$ for some field automorphism $\alpha$ of $M/N$ of order $t$ with $\pi(t)\not\subseteq \pi(M/N)$. Let $x\in \pi(t)$. It follows by \cite[Theorem A]{WhiteE} that $x$ is adjacent to every prime in $\pi(2^{2f}-1)$. But we know that $|\pi(2^{2f}-1)|\ge 5$ when $|\pi(S)|\ge 6$, a contradiction. Thus $|\pi(S)|=5 $. Now suppose that $\Delta(S)$ is isomorphic to Figure~\ref{fig:p}(a). Then it follows that $x$ together with $\pi(2^{2f}-1)$ forms a butterfly. Thus we must have that $\DG$ has either 7 or 9 vertices. Suppose that $\DG$ has seven vertices. Then we have that $|\mathcal{C}|=1$. Let $r\in \mathcal{C}$ and let $\theta\in \Irr(N)$ be such that $r|\theta(1)$. Then $\theta$ cannot be $M$-invariant since that would imply that $\theta$ extends to $M$ which implies that $r$ is adjacent to all the primes in $\pi(S)$, a contradiction. Let $M_\theta$ be the stabilizer of $\theta$ in $M$. Then by Lemma~\ref{conclusion}, we must have that $|M:M_\theta|$ is divisible by three primes which implies that $\{r,2\}\cup \pi(2^f+\epsilon), \epsilon\in \{-1,1\}$ forms a complete cubic. This contradicts the fact that $\DG$ does not contain a $K_4$. So we may assume that $|\rho(G)|=9$. In this case we have that $|\mathcal{C}|=3$ which in turn implies that there is $r,l\in \mathcal{C}$ with $r\sim l$. If we let $\theta(1)\in \cd(N)$ be the degree affording the edge, then we see that $\theta$ cannot extend to $M$. Thus we obtain that $M_\theta<M$ and by Lemma~\ref{conclusion}, $\pi(M:M_\theta)\cup \{r,l\}$ form a $K_5$, a contradiction.
\end{proof}

\begin{lem}\label{9vert}
Assume Hypothesis~\ref{hyp}. Let $|\rho(G)|=9$. Then $G$ does  not exist.
\end{lem}

\begin{proof}

{\bf Case 1: $|\pi(S)|= 8$}

Let $r\in \mathcal{C}$ and let $\theta\in\Irr(N)$ be such that $r|\theta(1)$.
This case is easy to see that $r\in \mathcal{C}$ together with four primes in $\pi(S)$ form a $K_5 $ whether $\theta$ extends to $M$ or not, a contradiction. 

{\bf Case 2: $|\pi(S)|= 7$}

By Lemma~\ref{onAS}, we must have that $|\pi(S)|=|\pi(G/N)| $. We have that $|\mathcal{C}|=2$. Now, observe that $3\le |\pi(2^f+\epsilon)\cup \{2\}|\le 5$. Let $\theta, \phi\in \Irr(N)$ be such that $r|\theta(1), l|\phi(1)$ where $\mathcal{C}=\{r,l\}$. Let $M_\theta=I_M(\theta)$ and $ M_\phi=I_M(\phi)$. Since $|\pi(S)|=7$, $\theta, \phi$ cannot extend to $M$ so we may assume that $M_\theta,M_\phi<M$. By Lemma~\ref{conclusion}, we have that $\pi(M:M_\theta)=\pi(M:M_\phi)$ and both contain three elements. This would imply that $\DG$ contains two $K_4$'s which intersect at vertex 2. This does not form a subgraph of any 4-regular graph with nine vertices.

{\bf Case 3: $|\pi(S)|= 6$}

In this case we have that $|\pi(G/N)|=|\pi(S)|$ by Lemma~\ref{onAS}. This implies that $|\mathcal{C}|=3$ and thus spans an edge. Let $\{r,l,h\}=\mathcal{C}$ such that $r\sim l$. Let $\theta, \phi\in \Irr(N)$ be such that $rl|\theta(1), h|\phi(1)$. Let $M_\theta=I_M(\theta)$ and $ M_\phi=I_M(\phi)$. Since $|\pi(S)|=6$, $\theta, \phi$ cannot extend to $M$ so we may assume that $M_\theta,M_\phi<M$. Observe that $|\pi(2^f+\epsilon)\cup\{2\}|\in \{2,3,4,5\}, \epsilon\in \{-1,1\}$. By Lemma~\ref{conclusion}, we have that $\pi(M:M_\theta)=\pi(M:M_\phi)$ and both contain two elements. In this case we obtain that $\DG$ must be isomorphic to Figure~\ref{fig:4}(a) since we obtain two $K_4$'s. Relabel Figure~\ref{fig:4}(a) as in the Figure below:

\begin{figure}[htb!]\centering
\begin{tikzpicture}
	\vertex (w) at (10:1.6) [label=right:$l$]{};
	\vertex (u) at (50:1.6) [label=right:$r$]{};
	\vertex (v) at (90:1.6) [label=above:$s$]{};
	\vertex (x) at (130:1.6) [label=left:$2$]{};
	\vertex (y) at (170:1.6) [label=left:$h$]{};
	\vertex (a) at (210:1.6) [label=left:$p_1$]{};
	\vertex (b) at (250:1.6) [label=left:$p_2$]{};
	\vertex (c) at (290:1.6) [label=below:$p_3$]{};
	\vertex (d) at (330:1.6) [label=right:$p_4$]{};
	\path 
 		(d) edge (a)
 		(d) edge (c)
 		(d) edge (w)
 		(d) edge (b)
 		(w) edge (u)
 		(w) edge (v)
 		(w) edge (x)
 		(u) edge (x)
 		(u) edge (c)
 		(v) edge (y)
 		(b) edge (y)
 		(v) edge (x)
 		(x) edge (y)
 		(a) edge (y)
 		(b) edge (a)
 		(u) edge (v)
 		(b) edge (c)
 		(a) edge (c)
	;
\draw (0,-2) node [below] {(a)} circle (0);
\end{tikzpicture}
\caption{Relabelled graph}
\label{fig:relabel}
\end{figure}
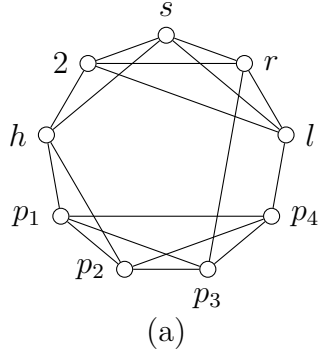\FloatBarrier

Where $\pi(2^f-\epsilon)=\{p_i\}_{i=1}^4$ and $\pi(2^f+\epsilon)=\{s\}$. We observe that $r\sim p_3$ and $l\sim p_4$. By previous arguments we obtain a contradiction since the pairs have no common neighbors. % in which case we may conclude that $p_3,p_4\in \rho(N)$ since $r,l\notin \pi(G/N)$. Now we observe that $\{r,h,p_4\}$ are three primes in $\rho(N)$ whic do not span an edge, a contradiction. 

{\bf Case 4: $|\pi(S)|=5$}

First suppose that $|\pi(G/N)|>|\pi(S)|$. Then by Lemma~\ref{onAS}, $\Delta(S)$ is isomorphic to Figure~\ref{fig:p}(c). Let $t\in \pi(G/N)\setminus \pi(S)$. Then it follows that $t$ is adjacent to every member of $\pi(2^{2f}-1)$ which obtains a $K_4$. Now, observe that $|\mathcal{C}|=3$. Let $ r,l,h\in \mathcal{C}$ and define $\theta, \phi$ in $\Irr(N)$ and $M_\theta, M_\phi$ as defined in the previous case. Let $|\pi(2^f-\epsilon)|=3$ and $|\pi(2^f+\epsilon)|=1$, $\epsilon\in \{-1,1\}$. Neither of $\theta, \phi$ extends to $M$ and so we must have $|\pi(M:M_\theta)|=2=|\pi(M:M_\phi)$ since otherwise we would have a $K_5$ or two $K_4$'s whose intersection is nonempty. This will as well result in either $\deg(2)$ or $\deg(s)$ having degree at least 5, a contradiction. Thus we may assume that $|\pi(S)|=|\pi(G/N)|$. Now we have that $|\mathcal{C}|=4$ which implies that $\mathcal{C}$ must span at least two edges. In this case we have that $\Delta(S)$ is one of the graphs in Figure~\ref{fig:p}(a) or (c).  Both cases would result in a $K_5$ or two non-disjoint $K_4$'s, a contradiction.

{\bf Case 5: $|\pi(S)|=9$}

 Suppose that $|\pi(S)|=9$. Then $\Delta(S)$ must have two $K_4$'s. So it suffices to consider the case when $\DG$ is isomorphic to Figure~\ref{fig:4}(a). Considering Figure~\ref{fig:relabel} above, we may consider an edge $r\sim p_3$ and obtain a contradiction as in Case 3 above.
\end{proof}

\begin{lem}\label{8vert}
Assume Hypothesis \ref{hyp}. Let $\DG$ be isomorphic to Figure~\ref{fig:3}. Then $\pi(2^f+1)$ and $\pi(2^f-1)$ belong to different $K_4$'s.
\end{lem}

\begin{proof}
It suffices to show that the case holds for $|\pi(S)|=5,6$. 

{\bf Case 1: $|\pi(S)|=5$}

First we suppose that $|\pi(G/N)|>|\pi(S)|$ and let $t\in \pi(G/N)\setminus \pi(S) $. Then we have that $t$ is adjacent to all primes in $\pi(2^{2f}-1)$. Since $\Delta(S)$ is isomorphic to Figure~\ref{fig:p}(c), we have that $\pi(2^f+\epsilon)\cup \{t\}, \epsilon\in \{-1,1\}$ form a $K_4$ where $|\pi(2^f+\epsilon)|=3$. This implies that $\pi(2^f-\epsilon)$ is contained in the other $K_4$. Now suppose that $|\pi(S)|=|\pi(G/N)|$. Then $|\mathcal{C}|=3$ and thus there is $r,l\in \mathcal{C}$ such that $r\sim l$ in $\Delta(N)$. Let $\theta\in \Irr(N)$ be such that $rl|\theta(1)$. Since $|\pi(S)|= 5$, we have that $\theta$ does not extend to $M$. Let $M_\theta=I_M(\theta).$ By Lemma~\ref{conclusion}, we must have that $\pi(M:M_\theta)=\pi(2^f-\epsilon)\cup\{2\}$ where $|\pi(2^f-\epsilon)|=1$. This implies that $\{r,l,2,h\}, h=\pi(2^f-\epsilon)$ form a $K_4$. This implies that $\pi(2^f+\epsilon)$ belongs to the other $K_4$ of $\DG$.

{\bf Case 2: $|\pi(S)|=6$}

In this case we must have that $|\pi(2^f+\epsilon)|=3,4$ and $|\pi(2^f-\epsilon)|=2,1$ in that order. It suffices to consider the case when $|\pi(2^f+\epsilon)|=3$ and $|\pi(2^f-\epsilon)|=2$. Let $r\in \mathcal{C}$ and let $\theta\in \Irr(N)$ be such that $r|\theta(1)$. Then again we observe that $\theta$ does not extend to $M$ and thus $M_\theta=I_M(\theta)<M$. By Lemma~\ref{conclusion}, we must have that $\pi(M:M_\theta)=\pi(2^f-\epsilon)\cup \{2\}$ since otherwise we obtain a $K_5$. Thus we have that $\{r,2\}\cup \pi(2^f-\epsilon)$ forms a $K_4$. This implies that $\pi(2^f+\epsilon)$ is in the other $K_4$.
\end{proof}

\begin{rem}
\label{pis6}
In the case $|\pi(S)|=6$, with $|\pi(2^f+\epsilon)|=3$ and $|\pi(2^f-\epsilon)|=2$ as described in the proof of Lemma~\ref{8vert} above, we can show that this case does not occur. We have obtained that  $\pi(2^f-\epsilon)$ and $\pi(2^f+\epsilon)$ are in different $K_4$'s. Moreover, in the proof, we observe that $\pi(2^f-\epsilon)\cup \{r,2\}$ forms one of the $K_4$'s in $\DG$. Which implies that the remaining vertex which is a prime in $\mathcal{C}$ is in the same $K_4$ as $\pi(2^f+\epsilon)$. Let $l$  be the remaining vertex and let $\alpha(1)\in \cd(N)$ be such that $l|\alpha(1)$. Let $M_\alpha=I_M(\alpha)$. Observe that $\alpha$ cannot extend to $M$ and thus $M_\alpha<M$.  By Lemma~\ref{conclusion} we must have that $\pi(M:M_\alpha)\in \{\{\pi(2^f+\epsilon)\cup \pi(2^f-\epsilon)\}, \{\pi(2^f+\epsilon)\cup \{2\}\}\}$. Either way $\DG$ will contain a $K_5$.
\end{rem}

\begin{lem}\label{8vertfinal}
Assume Hypothesis \ref{hyp}. Let $\DG$ be isomorphic to Figure~\ref{fig:3}. Then $G$ does not exist.
\end{lem}

\begin{proof}
Let $\pi(2^f-\epsilon)=\{p_i\}_{\forall i}$ and let $\pi(2^f+\epsilon)=\{q_\ell\}_{\forall \ell}$. By Lemma~\ref{8vert}, there is a $p_i$ and a $q_\ell$ in $\DG$ such that $p_i\sim q_\ell$ and that $\{p_i, q_\ell\} $ does not belong to any triangle in $\DG$. Let $\chi\in \Irr(G)$ be such that $p_iq_\ell|\chi(1)$. Let $\upxi\in \Irr(N)$ be an irreducible constituent of $\chi_N$. Since $p_i\not\sim q_\ell$ in $\Delta(S)$, we must have that $\upxi\neq 1_N$. We observe that $\pi(\chi(1))=\{p_i,q_\ell\}$. Let $M_\upxi=I_M(\upxi)$. Suppose that $\upxi$ is not $M$-invariant. Then we must have that $|M:M_\upxi|$ is divisible by all primes in two of the sets $\{2\}, \pi(2^f-\epsilon), \pi(2^f+\epsilon)$. This will obtain that $\pi(M:M_\upxi)$ contains three primes, or contains 2. Since $p_i, q_\ell$ are chosen in such a way that none equals to 2, we obtain a contradiction. Thus we may assume that $M_\upxi=M$. Since the Schur multiplier of $M/N$ is trivial, we have that $\upxi$ extends to $M$. By Gallagher's theorem, we have that $$\chi(1)\in \{\upxi(1), 2^f\upxi(1), (2^f-\epsilon)\upxi(1), (2^f+\epsilon)\upxi(1)  \} $$We deduce that $\chi(1)=\upxi(1)$ or $(2^f-\epsilon)\upxi(1)$ in which case we obtain that $q_\ell|\upxi(1)$. But  $(2^f+\epsilon)\upxi(1)$ and $(2^f-\epsilon)\upxi(1)$ which implies that $q_\ell$ is connected to all the primes in both $\pi(2^f+\epsilon)$ and $\pi(2^f-\epsilon)$, a contradiction.

\end{proof}

\begin{lem}\label{7vert}
Assume Hypothesis \ref{hyp}. Let $\DG$ be 4-regular with 7 vertices. Then $G$ does not exist. 
\end{lem}

\begin{proof}
By Lemmas~\ref{4primes}, \ref{psl2q} and Lemma~\ref{lem:others} we must have that $S\cong \PSL_2(2^f)$ and $5\leq |\pi(S)|\le 7$. 

{\bf Case 1: $|\pi(G/N)|=7$}

Suppose that $|\pi(S)|=\pi(G/N)=7$
Since $\DG$ does not contain a $K_4$, we have that $\Delta(S)$ is two disconnected triangles and an isolated vertex. We choose two disconnected triangles in Figure~\ref{fig:1}. Label them $\pi(2^f-1)=\{p_1,p_2,p_3\}$ and $\pi(2_f+1)=\{q_1,q_2,q_3\}$ so that $2=q_4$. The graph is vertex transitive and therefore whichever triangles we pick we obtain the same conditions. Observe that $p_1\sim q_1$ and this is not the case in $\Delta(G/N)$ whether $M/N<G/N$ or not. In fact if $t\in \pi(G:M)$ then we obtain that $\deg(t)\ge 5$. So we may assume that $M/N=G/N$. Let $\chi\in \Irr(G)$ be such that $q_1p_1|\chi(1)$ and let $\varphi\in \Irr(N)$ be such that $[\chi_N,\varphi]\neq 0$. We observe that $\varphi$ is nontrivial. Suppose that $I=I_G(\varphi)=G$, then we have that $\varphi$ is extendible to a $\varphi_0\in \Irr(G)$. In this case we obtain that $$\chi(1)\in \{\varphi_0(1)=\varphi(1), 2^f\varphi(1), (2^f-1)\varphi(1), (2^f+1)\varphi(1)   \}=\cd(G|\varphi).$$
In whichever case we obtain that $q_1p_1|\varphi(1)$. It is easy to observe from the degrees in $\cd(G|\varphi)$ that this case cannot occur. Therefore we we must have that $I<G$. By Lemma~\ref{McVey} and Lemma~\ref{conclusion}, we have that $|G:I|$ is divisible by at least three primes. A case which would result into a $K_4$ as a subgraph. Whichever two triangles we choose in Figure~\ref{fig:2}, we obtain the same conclusion.

{\bf Case 2: $|\pi(G/N)|=|\pi(S)|=6$}

In this case we obtain that $\Delta(S)$ is diconnected with one triangle, a path with two vertices and an isolated vertex. Let $r\in \mathcal{C}\subseteq\rho(N)\setminus\pi(G/N)$ and let $\lambda\in \Irr(N)$ be such that $r|\lambda(1)$. Let $I=I_M(\lambda)$. Then it follows that $I<M$, since otherwise $\lambda$ extends to $M$ and $\deg(r)=6$. Now, by Lemma~\ref{McVey} and Lemma~\ref{conclusion}, we obtain that $\DG$ contains a $K_4$, a contradiction.

Now we may suppose that $|\pi(S)|=5$ and $|\pi(G/N)|=7$. We skipped the case when $|\pi(S)|=6$ since it will obtain a $K_4$. So now we have  $t,w\in \pi(G/N)\setminus\pi(S)$. By \cite[Theorem A]{WhiteE}, we have that $w\sim t$. Since both $w$ and $t$ are adjacent to all the four primes in $\pi(2^{2f}-1)$, we obtain that $\deg(w), \deg(t)\ge 5$, a contradiction.

{\bf Case 3: $|\pi(G/N)|=6>|\pi(S)|=5$}

By Lemma~\ref{onAS}, we have that $\Delta(S)$ is not isomorphic to Figure~\ref{fig:p}(a). Suppose that $\Delta(S)$ is isomorphic to Figure~\ref{fig:p}(c), then by \cite[Theorem A]{WhiteE} we have that $\Delta(G/N)$ contains a $K_4$, a contradiction.

{\bf Case 4: $|\pi(G/N)|=|\pi(S)|=5$}

If $\Delta(S)$ is isomorphic to Figure~\ref{fig:p}(c). Let $\mathcal{C}=\{r,l\}$ we may assume that $r\not\sim l$ since otherwise we obtain a $K_4$, contradiction. Let $\lambda_1, \lambda_2\in \Irr(N)$ be such that $r|\lambda_1(1)$ and $l|\lambda_2(1)$. We must obtain that $|M:I_M(\lambda_1)|$ and $|M:I_M(\lambda_2)|$ are both divisible by only two primes $\{s,2\}$ where $s=(2^f-\epsilon)$, $\epsilon\in \{-1,1\}$.  Othewise $\DG$ would contains a $K_4$. In Figure~\ref{fig:2}, choose any two nonadjacent vertices and label them $r$ and $l$. There is only way to choose them so that the remaining vertices span a triangle. Let $\{r,l\}=\{p_2,q_4\}$  and let $\pi(2^f+\epsilon)=\{x,y,z\}=\{q_1,q_2,q_3\}$. We must have that $|M:I_M(\lambda_1)|=\{s,2\}=\{p_1,p_3\}=|M:I_M(\lambda_2)|$. Observe that $p_2\sim q_2$ in $\DG$. Let $\chi\in \Irr(G)$ be such that $p_2q_2|\chi(1)$ and let $\varphi\in \Irr(N)$ be such that $[\chi_N,\varphi]\neq 0$. Then $\varphi\neq 1_N$. Suppose that $\varphi$ extends to $G$, then we must have that $$\chi(1)\in \{\varphi(1), 2^f\varphi(1), (2^f-1)\varphi(1), (2^f+1)\varphi(1)  \}.$$
We observe that $p_2|\varphi(1)$. This implies that $\deg(p_2)\ge 5$ $(2^f-1)\varphi(1), (2^f+1)\varphi(1)\in \cd(M|\varphi)$, a contradiction. So we may suppose that $I=I_M(\varphi)<M$. By Lemmas~\ref{McVey} and \ref{conclusion}, we may assume that $|M:I|$ is divisible by two primes $\{2,s\}=\{p_3,p_1\}$ which implies that $p_2q_2|\varphi(1)$ and thus $\{p_2,p_1,p_3,q_2\}$ form  a $K_4$.

Suppose that $\DG$ is isomorphic to Figure~\ref{fig:1} and choose vertices $r$ and $l$ such that they do not span an edge such that the remaining vertices span a triangle. The graph is vertex transitive so we choose only one combination. Let $\pi(2^f+\epsilon)=\{x,y,z\}=\{q_1,q_2,q_3\}, \pi(2^f-\epsilon)=\{s\}=\{p_2\}$ and $2=p_3$, we are left with $\{r,l\}=\{p_1,q_4\}$. By previous argument we obtain a contradiction.

Now, we may suppose that $\Delta(S)$ is isomorphic to Figure~\ref{fig:p}(a). We may assume that $M/N<G/N$. Let $r,l\in \mathcal{C}$. Then we must have that $r$ and $l$ do not span an edge. Let $\gamma_1, \gamma_2\in \pi(N)$  be such that $r|\gamma_1(1)$ and $l|\gamma_2(1)$. Let $T_1=I_M(\gamma_1)$ and $T_2=I_M(\gamma_2)$. Suppose that $T_1=M$, then we have that $\gamma_1$ is extendible to an irreducible character $\gamma_{1_0}\in \Irr(M)$. This inplies that $r$ is adjacent to all the primes in $\pi(S)$ by Gallagher's Theorem. This however means that $\deg(r)\ge 5$, a contradiction. Thus we may assume that $T_1<M$. By Lemmas~\ref{McVey} and \ref{conclusion}, we have that $\DG$ contains a $K_4$, a contradiction. 
\end{proof}

\begin{lem}\textup{\cite[Lemma 3.1]{McVey}}\label{SchurM}
Lemma 3.1. Let $N\unlhd G$ be groups with $G/N \cong \PSL_2(q)$ for some prime-power $q\ge5$.
Suppose the character $\theta \in \Irr(N)$ is $G$-invariant but does not extend to $G$. Then, $q$ is odd and  one of the following holds:\begin{enumerate}
\item[(i)] $\cd(G|\theta) = {(q - 1)\theta(1), (q + 1)\theta(1), (q - \epsilon)/2}$ where $\epsilon = (-1)^{(q-1)/2}$,
\item[(ii)] $q = 9$ and $6\theta(1), 15\theta(1) \in \cd(G|\theta)$.
\end{enumerate}

\begin{hyp}\label{hyp2}
$G$ is a nonsovable group whose prime graph is 4-regular with 6 vertices. Let $N\unlhd G$ be the solvable radical. $G/N$ is almost simple with socle $M/N\cong S$ a nonabelian simple group. 
\end{hyp}

\begin{lem}\label{4primes2}
Assume Hypothesis \ref{hyp2}. Let $S\not\cong \PSL_2(2^f)$, $f>3$ and $|\pi(S)|=4$. Then $G$ does not exist. 
\end{lem}

\begin{proof}
It follows by \cite{LewisCo} that $\Delta(S)$ is connected or disconnected with two connected components. Suppose that $\Delta(S)$ is disconnected with two connected components. Then $S\cong \PSL_2(q)$ where $q$ is a power  of an odd prime $p$. It follows that $|\pi(q^2- 1)|=3$. $\Delta(S)$ contains a triangle if and only if $|\pi(q+\epsilon)|=3$, for some $\epsilon\in \{-1,1\}$. Otherwise $\Delta(S)$ contains a path and an isolated vertex. By choice of any subgraph of $\DG$ isomorphic to $\Delta(S)$ we obtain that $\mathcal{C}$ contains two adjacent vertices, say $r,l$. Let $\chi(1)\in\cd(G)$ be such that $rl|\chi(1)$. Let $\gamma\in \Irr(N)$ be an irreducible constituent of $\chi_N$. Then we have that $\gamma\neq 1_N$. Since $r,l\not\in \pi(G/N)$, we see that $rl|\gamma(1)$. By Lemma~\ref{TVtri} $\varphi(1)/\gamma(1)$ is divisible by two distinct primes in $\pi(S)$ for some $\varphi\in \Irr(T|\gamma),$ where $ T=I_M(\gamma)$ or $\gamma$ is extendible to $M$ in which case $S\cong \PSL_2(8)$ or $\Alt_5$. The former implies that $\DG$ must contain a $K_4$, a contradiction. The latter does not occur. Now suppose that $|\pi(S)|<|\pi(G/N)|$, then by \cite{Hupp, ATLAS}, we must have that $|\pi(G/N)|=5$ and $\Delta(G/N)$ contains as a subgraph a square with one diagonal or a $K_4$ in the case where $\Delta(S)$ contains a triangle \cite[Theorem A]{WhiteE}. We may assume that the former occurs and in this case we must have that $|\pi(q\pm 1)|=2$. Let $r\in \rho(N)\setminus \pi(G/N)$ and let $\varphi\in \Irr(N)$ be such that $r|\varphi(1)$. Let $T_\varphi =I_M(\varphi)<M$. By Lemma~\ref{conclusion} we have that $|M:T_\varphi|$ is divisible by at least three primes or case (D) occurs. In this case it suffices to assume that $\{t\}=\pi(M:T_\varphi)$. By conclusion of case (D) we obtain that $\{t,r,2,3\}$ form a $K_4$ or $\{t,3,r\}, \{t,r,2\}$ and $\{t,r,5\}$ all form triangles. Which obtains a graph which is not a subgraph of Figure~\ref{fig:6vert}. So we may suppose that $T_\varphi=M$. By Gallagher's Theorem we obtain that $r$ is adjacent to all the primes in $\pi(S)$ or to all primes in $\pi(S)$ but $p$ by Lemma~\ref{SchurM}. Suppose that $\pi(q-1)=\{2,b\}$ and $\pi(q+1)=\{2,d\}$. we obtain that $b\sim p$ and $d\sim p$ in $\DG$. Let $\chi_1, \chi_2\in \Irr(G)$ be such that $bp|\chi_1(1)$ and $dp|\chi_2(1)$. Since $b\not\sim p$ and $d\not\sim p$ in $\Delta(G/N)$, we have that $\chi_1, \chi_2\neq 1_N$. Let $\varphi_1, \varphi_2$ be such that $[\chi_{1_N}, \varphi_1]\neq 0$ and $[\chi_{2_N},\varphi_2]\neq 0$. Let $T_1=I_M(\varphi_1)$ and $T_2=I_M(\varphi_2) $. Suppose that $T_1=M$. Then we must have that $$\chi_1(1)=\{\varphi_1(1), q\varphi_1(1), (q-1)\varphi_1(1), (q+1)\varphi, (q+\epsilon)\varphi_1(1)/2 \}$$
or $$\chi_1(1)=\{\varphi_1(1), (q-1)\varphi_1(1), (q+1)\varphi, (q-\epsilon)\varphi_1(1)/2 \}$$where $\epsilon=(-1)^{(q-1)/2}$. In both cases we obtain that $b$ is adjacent to all primes in $\pi(S)$. Similarly $d$ is also  adjacent to all primes in $\pi(S)$. This cannot occur. So we may assume that $T_1<M$. By Lemma~\ref{conclusion}, we have that $|M:T_1|$ is divisible by atleast three primes of $\pi(S)$. So we obtain that $b$ is contained in some triangle with two other primes in $\pi(S)$. This obtains that the degree of one of the primes in $\pi(S)$ is 5, a contradiction.

Suppose that $\Delta(S)$ is connected, then by \cite{Hupp, ATLAS, White}, we must have that $$S\in  \{\Alt_8\cong \PSL_4(2), \PSL_3(4), M_{11}, {}^2B_2(8), {}^2B_2(32)\}$$
It follows by \cite{ATLAS}, $|\pi(S)|$ is always equal to $|\pi(G/N)|$ except the case $S\cong {}^2B_2(8)$ in which case we must have $S=G/N$ since otherwise $\DG$ would have a $K_4$. Now, by structure of $\Delta(S)$ we will obtain that $\mathcal{C}$ contains two adjacent vertices. By Lemma~\ref{TVtri}, we obtain a contradiction.

\end{proof}

\end{lem}
\begin{lem}\label{6vert}
Assume Hypothesis \ref{hyp2}. Let $|\pi(S)|\neq 3$. Then $G$ does not exist.
\end{lem}

\begin{proof}

{\bf Case 1: $|\pi(G/N)|=4$ with $|\pi(S)|=4$}

In this case, it suffices to consider the case when $S\cong \PSL_2(2^f)$ by Lemma~\ref{4primes2}. It follows that $\Delta(S)$ is disconnected with two isolated vertices, say $\{2,s\}$ and a path with two vertices, say $\{x,y\}$ and that $\mathcal{C}$ contains two primes, say $\{r,l\}$. We may assume that $r\not\sim l$ since otherwise by Lemma~\ref{TVtri}, we must have that $\DG$ contains a $K_4$, a contradiction. Let $\lambda_1,\lambda_2\in \Irr(N)$ be such that $r|\lambda_1(1)$ and $l|\lambda_2(1)$. Let $T_1=I_M(\lambda_1)$ and $T_2=I_M(\lambda_2)$. Suppose that $T_1=T_2=M$. Then we have that $\lambda_1$ and $\lambda_2$ both extend to $M$ by Theorem~\ref{TSchur}. By Gallagher's Theorem we obtain that $$\cd(M|\lambda_i)=\{\psi_i(1), 2^f\psi_i(1),(2^f-1)\psi_i(1), (2^f+1)\psi_i(1) \}, i\in \{1,2\}.$$This obtains a graph isomorphic to Figure~\ref{fig:case4}(a). 

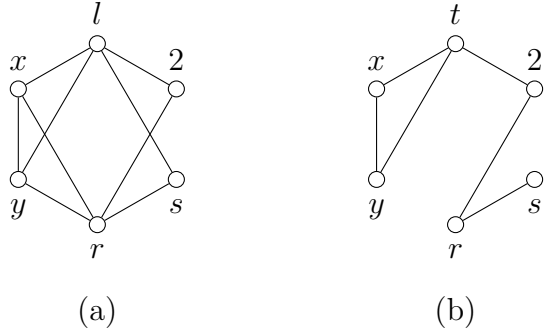
\begin{figure}[htb!]\centering
\begin{tikzpicture}
	\vertex (f) at (30:1.2)[label=above:$2$]{};
	\vertex (e) at (90:1.2)[label=above:$l$]{};
	\vertex (d) at (150:1.2)[label=above:$x$]{};
	\vertex (c) at (210:1.2)[label=below:$y$]{};
	\vertex (b) at (270:1.2)[label=below:$r$]{};
	\vertex (a) at (330:1.2)[label=below:$s$]{};
	\path 
		(b) edge (c)
		(b) edge (a)
		(b) edge (d)
		(b) edge (f)
		(c) edge (d)
		(c) edge (e)
%		(c) edge (a)
		(d) edge (e)
%		(d) edge (f)
		(e) edge (f)
		(e) edge (a)
%		(f) edge (a)
	;
	\draw (0,-2) node[below]{(a)} circle (0);
\end{tikzpicture}\hspace{2cm}
\begin{tikzpicture}
	\vertex (f) at (30:1.2)[label=above:$2$]{};
	\vertex (e) at (90:1.2)[label=above:$t$]{};
	\vertex (d) at (150:1.2)[label=above:$x$]{};
	\vertex (c) at (210:1.2)[label=below:$y$]{};
	\vertex (b) at (270:1.2)[label=below:$r$]{};
	\vertex (a) at (330:1.2)[label=below:$s$]{};
	\path 
%		(b) edge (c)
		(b) edge (a)
%		(b) edge (d)
		(b) edge (f)
		(c) edge (d)
		(c) edge (e)
%		(c) edge (a)
		(d) edge (e)
%		(d) edge (f)
		(e) edge (f)
%		(e) edge (a)
%		(f) edge (a)
	;
	\draw (0,-2) node[below]{(b)} circle (0);
\end{tikzpicture}
\caption{Possible subgraphs of 4-regular graph of order 6}
\label{fig:case4}
\end{figure}\FloatBarrier
We observe that $2\not\sim x$ in the subgraph above but $2\sim x$ in $\DG$ so we let $\gamma\in \Irr(G)$ be such that $2x|\gamma(1)$ and let $\zeta\in \Irr(N)$ be such that $[\gamma_N,\zeta]\neq 0$. Then we must have that $\zeta\neq 1_N$ since $2\not\sim x$ in $\Delta(G/N)$.  Suppose that $\zeta$ is $M$-invariant. Then it follows that $\zeta$ extends to $M$, in which case we have $$\gamma(1)\in \{\zeta(1), 2^f\zeta(1), (2^f-1)\zeta(1), (2^f+1)\zeta(1)  \}.$$ In this case we must have that $x|\zeta(1)$. By the fact that $2^f\zeta(1), (2^f-1)\zeta(1), (2^f+1)\zeta(1)\in \cd(M|\zeta)$, we have that $\{x,s\}$ and $\{x,2\}$ are edges and thus $\deg(x)=5$, a contradiction. Thus we may assume that $\zeta$ is not $M$-invariant. In this case we have by Lemma~\ref{conclusion} that $|M:I_M(\zeta)|$ is divisible by 2 and $s$ or there is a degree in $\cd(M|\zeta)$ divisible by three primes in $\pi(S)$, a contradiction. However, if the former occurs then we obtain that $\{2,s,x\}$ is a triangle implying $\deg(x)=5$, a contradiction. Now we may suppose that one of $T_1, T_2$ is properly contained in $M$. Without loss of generality,  let $T_1<M$. Then by Lemma~\ref{conclusion}, we have that $|M:T_1|$ is divisible by only two primes $2$ and $s$, or we obtain a $K_4$. But by assumption we must have that $r\sim x$ and $r\sim y$. This will not occur so now we may let $\chi\in \Irr(G)$  such that $rx|\chi(1)$. Let $\theta\in \Irr(N)$ be such that $[\chi_N,\theta]\neq 0$. Then we have that $\theta$ is nontrivial. Suppose that $I_M(\theta)=M$. Then we have that $$\chi\in \{\theta(1), 2^f\theta(1), (2^f-1)\theta(1), (2^f+1)\theta  \}.$$
Either way we obtain that $xr|\theta(1)$. This however obtains triangles $\{r,x,y\}$, $\{r,x,s\},\{r,x,2\}$, a contradiction. Thus we must have that $I_M(\theta)<M$. We may thus use Lemma~\ref{conclusion} to obtain that $\DG$ contains a $K_4$, a contradiction.

{\bf Case 2: $|\pi(G/N)|=5$}

Suppose that $|\pi(G/N)|=5$ and $S\cong \PSL_2(2^f)$ with $|S|=5$. Assume further that $\Delta(S)$ does not contain a triangle. Then $\Delta(S)$ is isomorphic to Figure~\ref{fig:p}(a).  Let $\mathcal{C}=\{r\}$. Let $\zeta\in \Irr(N)$ be such that $r|\zeta(1)$. Let $I=I_M(\zeta)$.  If $I=M$, then we have that $\zeta$ extends to $M$ and hence obtain that $r$ is adjacent to all primes in $\pi(S)$. This implies that $\deg(r)=5$, a contradiction. Now, suppose that $I<M$. By Lemma~\ref{conclusion} and Lemma~\ref{McVey}, we have that $|M:I|$ is divisible by at least three primes in $\pi(S)$ or case (D) occurs. In all the cases we  obtain  a $K_4$, a contradiction. 

Now assume that $\Delta(S)$ is isomorphic to Figure~\ref{fig:p}(c). and let $\zeta$ and $I$ be as defined. Suppose that $I<M$. Then again by Lemma~\ref{conclusion}, we must have that $|M:I|$ is divisible by only two primes 2 and $s=2^f+\epsilon, \epsilon\in \{-1,1\}$ (This conclusion is obtained due to the fact that the Frobenius group is of order $2^f(2^f-\epsilon)$ and if $(2^f+\epsilon)$ is composite, then we will have that $2|M:I|$ divides some degree in $\cd(I|\zeta)$ by Ito-Michler's Theorem). By \cite[Hauptsatz II.8.27]{Hp} we have that $I/N$ is a Frobenius group, a cyclic group of order $(2^f-\epsilon)$ or a dihedral group of order $2(2^f-\epsilon)$. But $r$ is connected to two more primes in $\pi(S)$. This will imply that the two more vertices adjacent to $r$ are in $\rho(N)$ and are indeed connected in $\Delta(N)$. Let $\pi(2^f-\epsilon)=\{d,b\}$ be the two primes in $\pi(S)$ adjacent to $r$ and let $\lambda_1, \lambda_2\in \Irr(N)$ be such that $rb|\lambda_1(1)$ and $rd|\lambda_2(1)$. Let $T_{\lambda_1}=I_M(\lambda_1)$ and $T_{\lambda_2}=I_M(\lambda_2)$. Suppose that $T_{\lambda_1}=M$, Then by Gallagher's theorem we have that $r$ is adjacent to all the primes in $\pi(S)$ implying that $\deg(r)=5$. So we must have $T_{\lambda_1}<M$. By Lemma~\ref{conclusion}, we have that $\{b,2,r,s\}$ forms a $K_4$ or $r$ is adjacent to all primes in $\pi(2^f-\epsilon)$ obtaining a $K_4$, a contradiction. 

Now suppose that $|\pi(S)|=4$ and $|\pi(G/N)|=5$. By \cite{Hupp, GAP}, we obtain that if $2\neq t\in\pi(G/N)\setminus\pi(S)$, then $\{t,x,y\}$ and $\{t,x\}$ are the only cliques in $\Delta(G/N)$. This means that 2 is still an isolated vertex.  Let $r\in \rho(N)\setminus \pi(G/N)\subseteq \rho(N)$ and let $\nu\in \Irr(N)$ be such that $r|\nu(1)$. Suppose that $I_M(\nu)=M$, then we obtain a subgraph isomorphic to Figure~\ref{fig:case4}(a) and following the arguments on case 1 we obtain a contradiction. Thus we must have that $I_M(\nu)<M$. In this case we must have that $|M:I_M(\nu)|$ satisfies case (A) and thus we must have that $\pi(M:I_M(\nu))=\{2,s\}$ obtaining a subgraph isomorphic to Figure~\ref{fig:case4}(b). Like in case 1 we see that $r\not\sim x$ and as argued we obtain a contradiction. 

Suppose that $S\cong \PSL_2(q), q>7$ a power of an odd prime $p, |\pi(S)|=5$ and $|\pi(S)|=|\pi(G/N)|$. We must have that $\Delta(S)$ is disconnected with two connected components one of which is an isolated vertex and the other is a 4-vertex graph with one triangle. This case we must have that $|\pi(q\pm 1)|\in \{2,3\}$. Let $r\in \mathcal{C}=\rho(G)\setminus \pi(G/N)\subseteq \rho(N)$ and let $\varphi\in \Irr(N)$. Let $T=I_M(\varphi)=M$. We must have that $\varphi$ extends to $M$ and $r$ is adjacent to every prime of $\pi(S)$ or the hypothesis of Lemma~\ref{SchurM} is satisfied. If the former occurs, we obtain a contradiction since $\deg(r)= 5$. If the latter occurs, then by Lemma~\ref{SchurM}, $(q-1)\varphi(1), (q+1)\varphi(1)\in\cd(M|\theta)$. Since one of $\pi(q\pm 1)$ is a triangle, we obtain a $K_4$, a contradiction. So we may assume that $|M:T|>1$. By Lemma~\ref{conclusion} we must have that $|M:T|$ is divisible by at least three primes, or case (D) occurs. In both cases, we obtain a $K_4$, a contradiction. The case when $|\pi(S)|<|\pi(G/N)| $ and $|\pi(S)|=4$ has been handled in Lemma~\ref{4primes2}.

{\bf Case 3: $|\pi(G/N)|=6$}

In this case we must have that $M/N\cong J_1$ or $\PSL_2(q), q$ a power of a prime $p$. Suppose that $S\cong J_1$. Observe that $M/N=G/N$ by \cite{ATLAS}. Observe that $11\sim 5$ in $\DG$ but $11\not\sim 5$ in $\Delta(G/N)$. Let $\chi\in \Irr(G)$ be such that $5\cdot 11|\chi(1)$. Let $\theta\in \Irr(N)$ be an irreducible constituent of $\chi_N$. We observe that $\theta\neq 1_N$. Let $G_\theta=I_G(\theta)$. Suppose that $G_\theta=G$. Since the Schur multiplier of $J_1$ is trivial, it follows that $\theta$ extends to $\theta_0$ in $\Irr(G)$. Thus we have $$\chi(1)\in \{\theta_0(1)=\theta(1), 56\theta(1), 76\theta(1), 77\theta(1), 120\theta(1), 133\theta(1), 209\theta(1)  \}.$$ In whichever case we obtain that $5\cdot 11|\theta (1)$. In this case we obtain that both 5 and 11 are adjacent to all other primes in $\rho(G)$. Thus we must have that $G_\theta<G$. This implies that $G_\theta/N$ is contained in one of the maximal subgroups $H/N$ of $J_1$. We obtain that $|G:H|$ is divisible in all cases except when $H/N\cong 2^3:7:3$ a nonabelian group. We may suppose that $G_\theta/N\cong H/N$ in this case. Since it is nonabelian, then there is a degree in $\cd(G_\theta|\theta)$ divisble by either $2\theta(1), 3\theta(1) $ or $7\theta(1)$. In either case we obtain that there is a degree in $\cd(G|\theta)$ divisible by $5\cdot 11\cdot 19\cdot r\theta(1)$, $r\in \{2,3,7\}$. This is not however possible since we obtain a $K_4$.

Now suppose that $S\cong \PSL_2(q)$ for an odd prime $p$. Then $\Delta(S)$ is disconnected with two connected component. Since $\DG$ does not have a $K_4$, it follows that $|\pi(q\pm 1)|=3$. We also see that $|G:M|$ is a power of 2 since otherwise, we would obtain that the connected component of $\Delta(S)$ with 5 vertices contains at least two complete vertices. Which is not possible in construction of $\DG$. Thus we may assume that $\Delta(G/N)$ is isomorphic to $\Delta(S)$. Let $\chi\in \Irr(G)$ be such that $pt|\chi(1)$ for some $t\in \pi(q\pm 1)\setminus \{2\}$. Let $\theta$ be an irreducible constituent of $\chi_N$. Since $p\not \sim t$ in $\Delta(G/N)$, it follows that $\theta\not=1_N$. Let $M_\theta$ be the stabilizer of $\theta$ in $M$. Suppose that $M_\theta<M$.  By Lemma~\ref{conclusion}, we have that $|M:M_\theta|$ is divisible by all primes in two of the sets $\{p\}, \pi(q-1), \pi(q+1)$. In either case we obtain that $|M:M_\theta|$ is divisible by at least 3 primes in which case we obtain a $K_4$, a contradiction thus we must have that $M_\theta=M$. If $\theta$ is extendible to some $\theta_0$ in $\Irr(M)$, then we have that $$\chi(1)\in \{\theta_0(1)=\theta(1), (q-1)\theta(1), (q+1)\theta(1), (q+\epsilon)\theta(1)/2 \}$$ where $\epsilon=(-1)^{(q-1)/2}$.

If $\chi(1)=\theta(1)$ or any other we obtain that $p|\theta(1), t|\theta(1)$  or $tp|\theta(1)$. Since $(q-1)\theta(1), (q+1)\theta(1)\in \cd(M|\theta)$ and  we must have that $2\sim p$ or $t$ is connected to all primes in $\pi(q^2-1)$ which would imply that $\deg(2)\ge 5$ or $\DG$ contains a $K_4$, a contradiction. Thus $\theta$ is does not extend to $M$. By Lemma~\ref{SchurM} $$\chi(1)=\{\theta_0(1)=\theta(1), (q-1)\theta(1), (q+1)\theta(1), (q-\epsilon)\theta(1)/2 \}.$$
A similar argument obtains a contradiction.

We now suppose that $S\cong \PSL_2(2^f)$. Then we have that $|\pi(S)|=5,6$. Let $|\pi(S)|=6$. Then we must have that $\Delta(S)$ contains a triangle. Let $|\pi(2^f+\epsilon)|=3$ and $|\pi(2^f-\epsilon)|=2$, $\epsilon\in \{-1,1\}$. Observe that if $t\in \pi(G:M)$, then $t\in \pi(2^f+\epsilon)$. Otherwise, we obtain that $\Delta(G/N)$ will contain a $K_4$. Also observe that $2\not\in \pi(G:M)$ since $\deg(2)\geq 5$ by \cite{WhiteE}. So we note that $\Delta(G/N)$ has an isolated vertex 2. Then in the construction of $\DG$, we should not that 2 is adjacent to all but one prime in the set $\pi(2^{2f}-1)$. Let $l\in \pi(2^{2f}-1)$ be such that $2\sim l$. Let $\chi\in \Irr(G)$ be such that $2l|\chi(1)$, and let $\nu\in \Irr(N)$ be the irreducible constituent of $\chi_N$. It is not hard to note that $\nu\neq 1_N$. Let $I=I_M(\nu)=M$. Then $\nu$ is extendible to an irreducible character $\nu_0$ of $M$. Hence we must have $$\chi(1)\in \{\nu_0(1)=\nu(1), 2^f\nu(1), (2^f-1)\nu(1), (2^f+1)\nu(1) \} $$
If $\chi(1)=\nu(1)$ or $2^f\nu(1)$, we have that $l|\nu(1)$. If we choose $l$ such that $l\in \pi(2^f+\epsilon), \epsilon\in \{-1,1\}$ with $|\pi(2^f+\epsilon)|=3$, then by $\pi(2^f+1)\nu(1)$ and $\pi(2^f-1)\nu(1)$ we obtain that $\deg(l)\ge 5$, a contradiction.
We may thus suppose that $I<M$. By Lemma~\ref{conclusion} we obtain that $|M:I|$ is divisible by at least 3 primes with 2 among them. This however results into a $K_4$, since we chose $l$ to belong to the triangle in $\Delta(S)$, a contradiction.

Now we may suppose that $|\pi(S)|=5$. Let $2\neq t\in \pi(G/N)\setminus\pi(S)$. Then by \cite[Theorem A]{WhiteE}, we see that $\Delta(S)$ cannot contain a triangle and thus should be isomorphic to Figure~\ref{fig:p}(1). This implies that $\Delta(S)$ has two connected components (a butterfly and an isolated vertex). Again if we let $\vartheta\in \Irr(G)$ be such that $t s|\vartheta(1)$ for some $s\in \pi(2^{f}-1)$ and $t\in \pi(2^f+1)$. Let $\zeta\in \Irr(N)$ be an irreducible constituent of $\vartheta_N$. Observe that $\zeta\neq 1_N$. Let $I=I_M(\zeta)=M$. By previous argument we obtain that $ts$ are adjacent to all primes in $\pi(S)$ which implies that $\DG$ contains a $K_4$. So we may assume that $I<M$. In this case we obtain that $|M:I|$ is divisible by all primes in two of the sets $\{2\}, \pi(2^f\pm 1)$ which implies that $\DG$ contains a $K_4$, a contradiction.

\end{proof}

\begin{lem}
Assume Hypothesis~\ref{hyp2}. Let $|\pi(S)|=3$, then $G$ does not exist.
\end{lem}

\begin{proof}
By \cite{Hupp, ATLAS}, we have that  $|\pi(S)|=3=|\pi(G/N)|$. By any choice of three vertices in Figure~\ref{fig:6vert}, we have that the remaining vertices span an edge. Let $r,l\in \mathcal{C}=\rho(G)\setminus \pi(S)$ be such that $r\sim l$. Let $\vartheta\in \Irr(N)$ be such that $rl|\vartheta(1)$. Then by Lemma~\ref{TVtri}, $\vartheta$ is extendible to $M$ or $\psi(1)/\vartheta(1)$ is divisible by two distinct primes in $\pi(S)$. If $\vartheta$ is extendible to $M$, then $rl$ are adjacent to all the primes in $\pi(S)$. The subgraph obtained is not an induced subgraph of Figure~\ref{fig:6vert}. If $\psi(1)/\vartheta(1)$ is divisible by two distinct primes in $\pi(S)$, then we obtain a $K_4$. 
\end{proof}

\begin{lem}\label{lem:10}
Let $G$ be a finite nonsolvable group with a 4-regular  prime graph $\Delta(G)$  with more than 9 vertices. Then $G$ does not exist.
\end{lem}

\begin{proof}
Since $\DG$ is 4-regular, then it is clear that it is $K_5$-free. The proof follows from \cite[Theorem A]{Zainab}.
\end{proof}

\begin{proof}[Proof of Theorem~\ref{thm:MAIN}]
The proof follows from Lemmas~\ref{9vert}, \ref{8vert}, \ref{7vert} and \ref{6vert}.
\end{proof}

\end{document}